\documentclass[a4paper]{amsart}
\pdfoutput=1
\usepackage[utf8]{inputenc}
\usepackage[T1]{fontenc}
\usepackage{lmodern}
\usepackage{amssymb}
\usepackage[lite]{amsrefs}

\usepackage[utf8]{inputenc} 
\usepackage{mathtools,amsmath, amssymb, amsthm} 
\usepackage{hyperref} 
\usepackage{microtype}

\numberwithin{equation}{section}

\theoremstyle{plain}
\newtheorem{theorem}[equation]{Theorem}
\newtheorem{lemma}[equation]{Lemma}
\newtheorem{proposition}[equation]{Proposition}
\newtheorem{corollary}[equation]{Corollary}
\theoremstyle{definition}
\newtheorem{definition}[equation]{Definition}
\theoremstyle{remark}

\newtheorem{example}[equation]{Example}

\newcommand*{\nb}{\nobreakdash}

\newcommand*{\Cst}{\mathrm C^*}
\newcommand*{\Cont}{\mathrm C}
\newcommand*{\Cliff}{\mathrm {Cl}}
\newcommand*{\K}{\mathrm K}
\newcommand*{\Id}{\mathrm{Id}}
\DeclareMathOperator{\rank}{rank}
\DeclareMathOperator{\IM}{Im}
\DeclareMathOperator{\Aut}{Aut}
\DeclareMathOperator{\Endo}{End}
\newcommand*{\ima}{\mathrm i}
\newcommand*{\Gl}{\mathrm{Gl}}
\newcommand*{\Hom}{\mathrm{Hom}}

\newcommand*{\defeq}{\mathrel{\vcentcolon=}}
\newcommand*{\congto}{\xrightarrow\sim}

\newcommand{\N}{\mathbb N}
\newcommand{\Z}{\mathbb Z}
\newcommand{\R}{\mathbb R}
\newcommand{\C}{\mathbb C}
\newcommand{\Quat}{\mathbb H}
\newcommand{\Mat}{\mathbb M}
\newcommand{\Bound}{\mathbb B}
\newcommand{\T}{\mathbb T}
\newcommand{\Field}{\mathbb F}
\newcommand*{\Disk}{\mathbb D}
\newcommand*{\Sphere}{\mathbb S}

\newcommand{\conj}{\overline}

\DeclarePairedDelimiterX{\setgiven}[2]{\{}{\}}{#1\,{:}\,\mathopen{}#2}

\begin{document}

\title{On equivariant embeddings of $G$-bundles}

\author{Malkhaz Bakuradze}
\email{malkhaz.bakuradze@tsu.ge}
\address{A. Razmadze Mathematics Institute, Faculty of Exact and Natural Sciences, Tbilisi State University\\
Tbilisi\\
Georgia}

\author{Ralf Meyer}
\email{rmeyer2@uni-goettingen.de}
\address{Mathematisches Institut\\
Universit\"at G\"ottingen\\Bunsenstra\ss e 3--5\\
37073 G\"ottingen\\Germany}


\thanks{This work was supported by the Shota Rustaveli National Science Foundation of Georgia (SRNSFG) grant FR-23-779.}

\begin{abstract}
  For a compact group~\(G\), we give a sufficient condition for
  embedding one \(G\)-equivariant vector bundle into another one and
  for a stable isomorphism between two such bundles to imply an
  isomorphism.  Our criteria involve multiplicities of irreducible
  representations of stabiliser groups.  We also apply our result to
  ordinary nonequivariant vector bundles over the fields of
  quaternions, real and complex numbers and to ``real'' and
  ``quaternionic'' vector bundles.  Our results apply to the
  classification of symmetry-protected topological phases of matter,
  providing computable bounds on the number of energy bands required
  to distinguish robust from fragile topological phases.
\end{abstract}
\maketitle

\section{Introduction}

The \(\K\)\nb-theory of a compact space~\(X\) may be computed by hand
by classifying vector bundles over~\(X\) of increasing rank~\(k\),
starting with line bundles.  A vector bundle of rank~\(k\) produces
one of rank \(k+1\) by adding a trivial bundle.  For
finite-dimensional~\(X\), there is a threshold~\(k_0\) depending on
the dimension of~\(X\) such that this stabilisation map is bijective
for \(k\ge k_0\) (see~\cite{Husemoller:Fibre_bundles_3}).  The authors
recently extended this classical result to ``real'' and
``quaternionic'' vector bundles over a space with involution
(see~\cite{Bakuradze-Meyer:Iso_stable_iso}).  These are vector bundles
that carry a conjugate-linear automorphism that lifts the involution
on the base.  This result is interesting in connection with the
classification of topological phases of matter in physics.  Here a
certain vector bundle, called the Bloch bundle, is used as a
topological invariant of a quantum mechanical physical system.  It is
physically interesting to know whether or not this vector bundle is
trivial because this is equivalent to the existence of ``exponentially
localised Wannier functions'' (see
\cite{DeNittis-Lein:Wannier_zero_flux}*{Proposition~4.3}).  It is
usually much easier to decide whether the Bloch bundle is stably
trivial, that is, whether it becomes trivial after adding another
trivial bundle.  This means that its class in reduced \(\K\)\nb-theory
vanishes.  If the reduced \(\K\)\nb-theory is torsion-free, this
happens if and only if its Chern numbers vanish.  Thus it is useful to
know whether the triviality of the Bloch bundle follows from its
stable triviality.  Our results show that for vector bundles of
suitable rank, stable isomorphism is the same as isomorphism.  In
physics parlance, this says that all fragile topological phases are
already stable.

For a quantum system with time-reversal symmetry, the Bloch bundle
inherits the extra structure of a ``real'' or ``quaternionic'' vector
bundle, depending on whether the square of the time-reversal symmetry
is~\(\pm1\).  This provides a link between the result
in~\cite{Bakuradze-Meyer:Iso_stable_iso} and the study of topological
phases.  Instead of a time-reversal symmetry, a system may also have
more classical crystallographic symmetries, which lead to
symmetry-protected topological phases, an area of much current
activity in the physics community.  These are described through a
finite subgroup~\(G\) of the orthogonal group.  Then the Bloch bundle
becomes a \(G\)\nb-equivariant vector bundle.  Once again, the
question is physically relevant whether the Bloch bundle is trivial
once its class in the reduced equivariant \(\K\)\nb-theory is trivial.


Unfortunately, \(G\)\nb-equivariant \(\K\)\nb-theory does not behave
as nicely as ``real'' or ``quaternionic'' \(\K\)\nb-theory in this
respect.  The reason is that there may be several nonisomorphic
irreducible representations.  Each finite-dimensional representation
\(\pi\colon G\to \Gl(V)\) gives rise to a trivial \(G\)\nb-equivariant
vector bundle \(X\times V\) over~\(X\).  An embedding of \(X\times V\)
into a \(G\)\nb-vector bundle~\(E\) over~\(X\) exists automatically if
\emph{all} irreducible representations of stabiliser groups contained
in~\(V\) are also contained in~\(E\) with sufficient multiplicity.
The rank of~\(E\), however, is no longer enough to control this, and
so stabilisation results may fail.

The key result in \cites{Husemoller:Fibre_bundles_3,
  Bakuradze-Meyer:Iso_stable_iso} is that any vector bundle of
sufficiently high rank must contain the trivial bundle of rank~\(1\)
as a subbundle or, equivalently, as a direct summand.  This says that
the stabilisation map from vector bundles of rank~\(k\) to
rank~\(k+1\) is surjective for sufficiently high~\(k\).  A relative
version of this result also implies that this map is injective for
sufficiently high~\(k\).  Example~\ref{exa:involutions_and_circle}
below shows that this result breaks down for equivariant
\(\K\)\nb-theory, even for the group~\(\Z/2\) acting on the circle by
complex conjugation.  Nevertheless, the proof techniques
in~\cite{Bakuradze-Meyer:Iso_stable_iso} also give related results in
equivariant \(\K\)\nb-theory.  The main point of this note is to
record these results.

Our main result also contains the classical results about real,
complex and quaternionic vector bundles and about ``real'' and
``quaternionic'' vector bundles over spaces with involution as special
cases.  More generally, we may cover vector bundles over the fields
\(\R\), \(\C\) and the quaternion skew-field~\(\Quat\) with extra
symmetries, which may be linear or conjugate-linear in the cases
\(\C\) and~\(\Quat\).  A rather general setup for symmetries in
quantum physics is developed in~\cite{Freed-Moore:Twisted_matter}.
The key idea for this is to shift the multiplication by the imaginary
unit~\(\ima\) in a complex vector bundle into the group action on the
underlying real vector bundle.  Thus a \(G\)\nb-equivariant
\(\C\)\nb-vector bundle \(E\to X\) becomes equivalent to a
\(G\times\Z/4\)\nb-equivariant \(\R\)\nb-vector bundle with some extra
properties, namely, the subgroup~\(\Z/4\) acts trivially on the
base~\(X\) and the square of the generator of~\(\Z/4\) acts by
multiplication by~\(-1\) in each fibre.  If part of~\(G\) acts by
conjugate-linear maps, this merely replaces the product
\(G\times\Z/4\) by a semidirect product.  Thus we identify ``real''
and ``quaternionic'' vector bundles with equivariant \(\R\)\nb-vector
bundles over the dihedral group and the quaternion group of
order~\(8\) with some extra properties.  These extra properties do not
concern the vector bundle maps, so that all these types of bundles
behave exactly like \(G\)\nb-equivariant \(\R\)\nb-vector bundles for
suitable~\(G\).

For best results, however, it is important that all stabiliser groups
of points in the base space have a unique irreducible representation
that can occur in the relevant vector bundles.  Then the
multiplicities of irreducible representations in our conditions may be
replaced by ranks of vector bundles, giving a much better result.
This is what allows for the special results for ``real'',
``quaternionic'' and ordinary real, complex or quaternionic vector
bundles.  Using this fact, we show that the results
in~\cite{Bakuradze-Meyer:Iso_stable_iso} are special cases of our
results here.

In Section~\ref{sec:bundles}, we explain how we treat various kinds of
vector bundles with extra structure as \(G\)\nb-equivariant
\(\R\)\nb-vector bundles for suitable groups~\(G\).  In particular,
this covers ``real'' and ``quaternionic'' vector bundles, but also
\(G\)-equivariant vector bundles over the fields \(\C\) and~\(\Quat\).
In Section~\ref{sec:main}, we state our main results about the
existence of trivial subbundles and about unstabilising a stable
isomorphism of vector bundles, both over relative and absolute
\(G\)-CW-complexes.  We prove these results in
Section~\ref{sec:proofs_main}.  In Section~\ref{sec:ordinary_real}, we
specialise to ordinary vector bundles without group action and to
``real'' and ``complex'' vector bundles.  We see that classical
results and the results in~\cite{Bakuradze-Meyer:Iso_stable_iso} are
special cases of our main theorems.  In
Section~\ref{sec:subbundle_trivial}, we use our main theorems to embed
equivariant vector bundles into a trivial bundle with sufficiently
high multiplicities.  This gives an equivariant version of Swan's
Theorem, where we also control the size of the trivial vector bundle
that we need.  In Section~\ref{sec:crystallographic}, we briefly
recall how equivariant Bloch vector bundles over tori arise from
insulators with crystallographic symmetries.  To illustrate our main
result, we then restrict further to the case where the only symmetry
besides translations is a point reflection.  In this case, we can show
that stable isomorphism is always the same as isomorphism if the
dimension is at most~\(4\).  We also show by an example that trivial
direct summands need not always exist for \(\Z/2\)-equivariant vector
bundles over the circle.

\section{Equivariant bundles}
\label{sec:bundles}

We are going to describe equivariant complex or quaternionic vector
bundles with extra symmetries as equivariant real vector bundles.  If
\(k,n\in\N\), we denote the image of~\(k\) in \(\Z/n\) by \([k]_n\) or
just~\([k]\).

\begin{proposition}
  \label{pro:C-vector_bundle_rewrite}
  Let~\(X\) be a topological space and equip~\(X\) with the trivial
  action of the group~\(\Z/4\).  The category of \(\C\)\nb-vector
  bundles over~\(X\) is isomorphic to the full subcategory of
  \(\Z/4\)\nb-equivariant \(\R\)\nb-vector bundles over~\(X\) where
  \([2]\in \Z/4\) acts by multiplication by~\(-1\) in each fibre.
\end{proposition}

\begin{proof}
  Let \(p\colon E\to X\) be a complex vector bundle.  It is also an
  \(\R\)\nb-vector bundle.  Multiplication by~\(\ima\) in each fibre
  gives a fibrewise \(\R\)\nb-linear map~\(I\) with \(I^2=-1\).  Since
  \(I^2=-1\) implies \(I^4=1\), we may view~\(E\) as a
  \(\Z/4\)\nb-equivariant \(\R\)\nb-vector bundle with the extra
  property that~\(\Z/4\) acts trivially on the base space~\(X\) of the
  bundle and \([2]\in \Z/4\) acts by multiplication by~\(-1\) in
  each fibre.  Conversely, such a \(\Z/4\)\nb-equivariant
  \(\R\)\nb-vector bundle comes from a \(\C\)\nb-vector bundle where
  multiplication by~\(\ima\) is the action of \([1]\in \Z/4\).
  These constructions are inverse to each other and
  natural, that is, they form an isomorphism of categories.
\end{proof}

Next, we enrich the isomorphism of categories in
Proposition~\ref{pro:C-vector_bundle_rewrite} to complex vector
bundles with extra symmetries.

\begin{definition}
  \label{def:G-gamma_equivariant} 
  Let~\(G\) be a compact group.  Let~\(X\) be a \(G\)\nb-space.  Let
  \(\gamma\colon G\to \Z/2\) be a group homomorphism.  A
  \emph{\((G,\gamma)\)-equivariant complex vector bundle} over~\(X\)
  is a \(\C\)\nb-vector bundle \(p\colon E\to X\) with a
  \(G\)\nb-action on~\(E\) such that~\(p\) is \(G\)\nb-equivariant and
  the maps \(E_x \to E_{g x}\), \(v\mapsto g\cdot v\), for \(g\in G\)
  are \(\C\)\nb-linear if \(\gamma(g)=[0]\) and conjugate-linear if
  \(\gamma(g)=[1]\); that is, these maps are additive and satisfy
  \(g\cdot (\lambda v) = \lambda (g\cdot v)\) if \(\gamma(g)=[0]\) and
  \(g\cdot (\lambda v) = \conj{\lambda} (g\cdot v)\) if
  \(\gamma(g)=[1]\).
\end{definition}

The cyclic group~\(\Z/4\) has two automorphisms, the trivial one and
the automorphism \(x\mapsto -x\).  Use this to identify
\(\gamma\colon G\to\Z/2\) with a homomorphism to
\(\Aut(\Z/4) \cong \Z/2\) and form the semidirect product group
\[
  G' \defeq \Z/4 \rtimes_\gamma G.
\]
Let~\([k]\) also denote the image of \([k]\in \Z/4\) in~\(G'\).  If
\(g\in G\subseteq G'\), then \([1] g = g [1]\) if \(\gamma(g)=0\) and
\([1] g = g [3]\) if \(\gamma(g)=1\).  This implies \([2] g = g [2]\)
for all \(g\in G\).  So~\([2]\) is a central involution in~\(G'\).

\begin{example}
  \label{exa:G-gamma_equivariant}
  If~\(\gamma\) is trivial, then a \((G,\gamma)\)-equivariant complex
  vector bundle is the same as a \(G\)\nb-equivariant \(\C\)\nb-vector
  bundle.  In this case, \(G' = G\times \Z/4\).

  Let \(G = \Z/2\) and \(\gamma=\Id_{\Z/2}\).  Then a
  \((G,\gamma)\)-equivariant complex vector bundle is exactly the same
  as a ``real'' vector bundle as
  in~\cite{Bakuradze-Meyer:Iso_stable_iso}, with \([1]\in\Z/2\) giving
  the ``real'' involution on the total space of the bundle.  The
  resulting group~\(G'\) is isomorphic to the dihedral group~\(D_8\)
  with eight elements: the isomorphism maps the normal subgroup
  \(\Z/4 \subseteq G'\) onto the rotation subgroup in~\(D_8\) and it
  maps the generator of~\(G\) to a reflection in~\(D_8\).
\end{example}

\begin{proposition}
  \label{pro:twisted_complex_vector_bundle}
  Let \(X\) be a \(G'\)\nb-space where the subgroup~\(\Z/4\) acts
  trivially.  Then a \(G'\)\nb-equivariant \(\R\)\nb-vector bundle
  over~\(X\) with the extra property that~\([2]\) acts by
  multiplication by~\(-1\) in each fibre is the same as a
  \((G,\gamma)\)-equivariant \(\C\)\nb-vector bundle, and a
  \(G\)\nb-equivariant \(\C\)\nb-vector bundle map is the same as a
  \(G'\)\nb-equivariant \(\R\)\nb-vector bundle map.
\end{proposition}

\begin{proof}
  We define the complex structure on~\(E\) by letting the imaginary
  unit act like \([1] \in G'\).  This defines a \(\C\)\nb-vector
  bundle by Proposition~\ref{pro:C-vector_bundle_rewrite}.  The
  \(G'\)\nb-action is the same as this complex structure together with
  an action of~\(G\) by \(\R\)\nb-linear maps that satisfy
  \(g(\ima v) = \ima g(v)\) if \(\gamma(g)=[0]\) and
  \(g(\ima v) = -\ima g(v)\) if \(\gamma(g)=[1]\).  This says that~\(g\)
  acts by a \(\C\)\nb-linear map if \(\gamma(g)=[0]\) and by a
  conjugate-linear map if \(\gamma(g)=[1]\).  A direct computation shows
  that an \(\R\)\nb-vector bundle map is \(G\)\nb-equivariant and
  \(\C\)\nb-linear if and only if it is \(G'\)\nb-equivariant.
\end{proof}

Together with Example~\ref{exa:G-gamma_equivariant}, the proposition
shows that \(G\)\nb-equivariant complex and ``real'' vector bundles
may be treated as \(G'\)\nb-equivariant \(\R\)\nb-vector bundles with
some extra properties.  Namely, a certain subgroup \(G_0\subseteq G'\)
acts trivially on the base of the bundle and a certain element
\(t\in G_0\) acts as multiplication by~\(-1\) in each fibre.  It is
clear that the extra properties on these vector bundles are inherited
by direct sums and direct summands of vector bundles over the same
base space.  Therefore, for our problems of finding embeddings and
isomorphisms between equivariant vector bundles, the extra conditions
have no effect.

\begin{example}
  \label{exa:quaternion_group}
  This example explains how to treat ``quaternionic'' vector bundles
  (see~\cite{Bakuradze-Meyer:Iso_stable_iso}).  Let \(G = \Z/4\) and
  let \(\gamma\colon \Z/4\to\Z/2\) be the canonical projection.  By
  definition, a ``quaternionic'' vector bundle is the same as a
  \((G,\gamma)\)-equivariant complex vector bundle with the extra
  property that \([2]\in \Z/4\) acts trivially on the base space~\(X\)
  and by multiplication by~\(-1\) in the fibres of the bundle.  When
  all of~\(G\) acts trivially on the base, then this is the same as an
  \(\Quat\)\nb-vector bundle.  Now turn a \((G,\gamma)\)-equivariant
  vector bundle into a \(G'\)\nb-equivariant \(\R\)\nb-vector bundle
  as above.  The group~\(G'\) contains two copies of \(\Z/4\), and
  \([2]\) in each copy acts trivially on the base and by
  multiplication by~\(-1\) in each fibre on the total space.  So the
  action of~\(G'\) drops down to an action of the quotient group where
  these two elements of~\(G'\) are identified.  This group has eight
  elements and the presentation
  \[
    \langle a,b\mid a^4 = 1,\ b^2 = a^2,\ b^{-1} a b = a^{-1}\rangle
  \]
  where \(a\) and~\(b\) are the generators of the two copies
  of~\(\Z/4\).  This gives the quaternion group
  \(Q_8 = \{\pm1,\pm \ima, \pm j, \pm k\} \subseteq \Quat\).  Thus a
  ``quaternionic'' vector bundle over a space~\(X\) with involution is
  the same as a \(Q_8\)\nb-equivariant vector bundle over~\(X\) with
  the extra property that the generators \(a\) and~\(b\) act
  trivially and by the given involution on the base~\(X\),
  respectively, and \(a^2 = b^2\) acts as multiplication by~\(-1\).
\end{example}

For context, we recall that when~\(X\) is a compact space with
involution, then the Grothendieck group of the monoid of ``real''
vector bundles over~\(X\) is Atiyah's \(\mathrm{KR}^0(X)\), whereas
the Grothendieck group of the monoid of ``quaternionic'' vector
bundles over~\(X\) is Atiyah's \(\mathrm{KR}^4(X)\) (see
\cite{Atiyah:K_Reality}).

We have now identified several important categories of vector bundles
with a category of equivariant real vector bundles with the extra
property that certain elements act trivially on the base and a
particular element acts as multiplication by~\(-1\).  We encode the
latter property in a definition:

\begin{definition}
  Let~\(G\) be a compact group and let \(t\in G\) be an element of
  order~\(2\).  A \(G\)\nb-equivariant \(\R\)\nb-vector bundle over a
  space~\(X\) is called \emph{odd} if \(t\) acts as multiplication
  by~\(-1\) in each fibre; this forces~\(t\) to act trivially on the
  base.
\end{definition}

In particular, a representation \(\pi\colon G\to \Gl(V)\) on a vector
space~\(V\) is called odd if \(\pi(t) = -1\).  Any representation is a
direct sum of irreducible representations by Maschke's Theorem.  Since
direct sums and direct summands of odd representations remain odd, any
odd representation of~\(G\) is a direct sum of odd, irreducible
representations.  In some cases, there is a unique irreducible odd
representation.  Then a trivial odd \(G\)\nb-equivariant vector bundle
is determined uniquely by its rank.  This fact is a key prerequisite
for some of the results in~\cite{Bakuradze-Meyer:Iso_stable_iso}, in
particular, about finding trivial direct summands.

\begin{proposition}
  \label{pro:unique_odd}
  For the following pairs \((G,t)\), there is a unique odd irreducible
  representation:
  \begin{itemize}
  \item \((\Z/2,[1])\);
  \item \((\Z/4,[2])\);
  \item the dihedral group~\(D_8\) with \(t\in D_8\) being rotation
    by~\(\pi\);
  \item the quaternion group \(Q_8\subseteq \Quat\) with
    \(t=-1 \in\Quat\) or \(t=a^2=b^2\) in the presentation in
    Example~\textup{\ref{exa:quaternion_group}}.
  \end{itemize}
\end{proposition}

\begin{proof}
  First let \(G=\Z/2\) and \(t=[1]\).  The group~\(G\) has two
  irreducible representations, namely, the trivial character and the
  sign character \([k]\mapsto (-1)^k\).  The latter is the only odd
  one.  Next, let \(G=\Z/4\) and \(t=[2]\).  This group has exactly
  three irreducible real representations, namely, the trivial
  character, the \(\R\)\nb-valued character \([k]\mapsto (-1)^k\), and
  the \(\C\)\nb-valued character \([k] \mapsto \ima^k\).  Only the
  latter is odd.

  Next, let \(G=D_8\).  Then the rotation by~\(\pi\) is the only
  nontrivial central element in~\(D_8\).  The group
  \(D_8/\langle t\rangle = \Z/2\times \Z/2\) has four \(\R\)\nb-valued
  characters \(D_8\to \{\pm1\}\), giving four \(\R\)\nb-valued
  characters that all kill the element~\(t\), so that they are not
  odd.  The standard representation of~\(D_8\) in~\(\R^2\) is a
  \(2\)\nb-dimensional irreducible real representation, which is
  clearly odd.  These are all irreducible representations because the
  group has order \(8 = 4\cdot 1 + 2^2\).  Thus the only odd
  irreducible representation is the \(2\)\nb-dimensional one.

  Finally, let \(G=Q_8\).  Once again, \(t\) as above is the unique
  nontrivial central element and the quotient \(Q_8/\langle t\rangle\)
  is the Klein Four Group, giving four \(\R\)\nb-valued characters
  that all kill~\(t\).  The canonical inclusion
  \(Q_8\hookrightarrow \Quat\) gives another irreducible
  representation of quaternionic type.  Since \(8 = 4\cdot 1 + 2^2\),
  these are all irreducible real representations.  So the quaternionic
  representation of~\(Q_8\) is its only irreducible odd
  representation.
\end{proof}

The theory of real Clifford algebras provides many other groups with a
unique irreducible odd representation.  For \(p,q\ge0\), let
\(\Cliff_{p,q}\) be the Clifford algebra over~\(\R\) generated by
anticommuting elements \(e_1,\dotsc,e_{p+q}\) with \(e_j^2=1\) for
\(j=1,\dotsc,p\) and \(e_j^2=-1\) for \(j=p+1,\dotsc,p+q\).  Let
\(G_{p,q} \subseteq \Cliff_{p,q}\) be the subset of elements of the form
\(\pm e_{i_1} e_{i_2} \dotsm e_{i_\ell}\) with \(\ell\ge0\) and
\(1 \le i_1 < i_2 < \dotsb < i_\ell \le p+q\).  The relations of
\(\Cliff_{p,q}\) imply that~\(G_{p,q}\) is a subgroup.  It has
\(2^{p+q+1}\) elements.  Let \(t = -1 \in G_{p,q}\).  This is a
central involution in~\(G_{p,q}\).  Recall that \([k]_8\) denotes the
class of \(k\in\Z\) in \(\Z/8\).

\begin{proposition}
  \label{pro:Cliff_groups}
  Let \(p,q\ge0\) be such that
  \([p-q]_8 \in \{[0],[2],[3],[4],[6],[7]\}\).  Then~\(G_{p,q}\) has a
  unique irreducible odd representation.
\end{proposition}

\begin{proof}
  We claim that~\(G_{p,q}\) is generated as a group by the Clifford
  generators \(e_1,\dotsc,e_{p+q}\) unless \((p,q)=(0,0)\).  The
  Clifford generators belong to~\(G_{p,q}\) by definition, and the
  claim follows once we are able to write~\(-1\) as a product of them.
  This is trivial if \(q\ge1\) because then \(e_{p+q}^2 = -1\).  If
  \(q=0\), then we ruled out \(p=0\), and \(p=1\) is forbidden because
  then \([p-q]_8 = [1]_8\).  So \(p\ge2\).  Then
  \(e_1 e_2 e_1 e_2 = -e_1^2 e_2^2 = -1\).  Thus~\(-1\) belongs to the
  subgroup generated by \(e_1,\dotsc,e_{p+q}\) in all cases \((p,q)\)
  that satisfy our hypothesis except for \((p,q)=(0,0)\).  If
  \((p,q)=(0,0)\), then \(\Cliff_{0,0}=\R\) and
  \(G_{p,q} = \{\pm1\} \cong \Z/2\) with \(t=[1]_2\), and the claim
  follows from Proposition~\ref{pro:unique_odd}.  So we may disregard
  the case \(p=q=0\) from now on and assume that~\(G_{p,q}\) is
  generated by \(e_1,\dotsc,e_{p+q}\).

  Let~\(\varrho\) be an odd representation of the group~\(G_{p,q}\).
  Thus \(\varrho(-1) = -1\) and
  \(\varrho(e_j)^2 = \varrho(e_j^2) = 1\) for \(j=1\dotsc,p\) and
  \(\varrho(e_j)^2 = \varrho(e_j^2) = -1\) for \(j=p+1\dotsc,p+q\).
  Thus the linear maps \(\varrho(e_j)\) determine a representation of
  the Clifford algebra~\(\Cliff_{p,q}\).  The representation~\(\varrho\)
  is equal to the restriction of this representation of~\(\Cliff_{p,q}\)
  to \(G_{p,q} \subseteq \Cliff_{p,q}\) because~\(G_{p,q}\) is generated
  by the Clifford generators.  Thus the category of odd group
  representations of~\(G_{p,q}\) is equivalent to the category of
  modules over~\(\Cliff_{p,q}\).  So~\(G_{p,q}\) has a unique irreducible
  representation if and only if there is a unique simple module
  over~\(\Cliff_{p,q}\).

  It remains to determine when~\(\Cliff_{p,q}\) has a unique simple
  module.  Recall that
  \(\Cliff_{p+1,q+1} \cong \Mat_2(\Cliff_{p,q})\).  So
  \(\Cliff_{p+1,q+1}\) has a unique simple module if and only if
  \(\Cliff_{p,q}\) has one.  So the answer depends only on
  \([p-q]_8\).  Recall that \(\Cliff_{0,0} \cong \R\),
  \(\Cliff_{0,1} \cong \C\), \(\Cliff_{0,2} \cong \Quat\),
  \(\Cliff_{0,3} \cong \Quat \oplus \Quat\),
  \(\Cliff_{1,0} \cong \R\oplus \R\),
  \(\Cliff_{2,0} \cong \Mat_2(\R)\),
  \(\Cliff_{3,0} \cong \Mat_2(\C)\),
  \(\Cliff_{4,0} \cong \Mat_2(\Quat)\).
  As a result, \(\Cliff_{p,q}\) has a unique simple module if and only if
  $$[p-q]_8\in \{[0]_8,[2]_8,[3]_8,[4]_8,[6]_8,[7]_8\}.$$
\end{proof}

Proposition~\ref{pro:Cliff_groups} generalises
Proposition~\ref{pro:unique_odd} because \(\Z/2\cong G_{0,0}\),
\(\Z/4\cong G_{0,1}\), \(Q_8 \cong G_{0,2}\), and
\(D_8 \cong G_{1,1} \cong G_{2,0}\).  The groups in
Proposition~\ref{pro:Cliff_groups} all have order~\(2^k\) for some
\(k\in\N\).  In addition, their centre \(Z(G_{p,q})\) is \(\{1,-1\}\)
and the quotient \(G_{p,q}/Z(G_{p,q})\) is \((\Z/2)^{p+q}\) because
all the generators are involutions and commute up to signs.

\section{The main results}
\label{sec:main}

In this section, we formulate our main results, which are equivariant
\(\K\)\nb-theory versions of the results
in~\cite{Bakuradze-Meyer:Iso_stable_iso}.  We will later assume that
our vector bundles are odd with respect to a suitable central
involution.  Since this does not affect the vector bundle maps, our
main theorems below immediately imply results for odd vector bundles.
Implicitly, restricting to odd vector bundles means that only odd
irreducible representations occur in the conditions of our theorems.

Let~\(G\) be a compact Lie group.  Let \((X,A)\) be a relative
\(G\)\nb-CW-complex.  This means that~\(X\) is built from a closed
invariant subspace \(A\subseteq X\) by successively attaching
\(G\)\nb-equivariant cells of the form \(G/H \times \Disk^\ell\) for
subgroups \(H\subseteq G\) and \(\ell\in \N\).  For instance, if~\(X\)
is a smooth manifold with a smooth action of a compact group~\(G\),
then it admits an equivariant triangulation
by~\cite{Illman:Equivariant_triangulations}, and this makes it
homeomorphic to a \(G\)-CW-complex.

A cell of type \(G/H \times \Disk^\ell\) occurs in~\((X,A)\) if and
only if there is a point \(x\in X\setminus A\) whose stabiliser group
is exactly equal to~\(H\).  If the subgroups \(H\) and~\(L\) are
conjugate, then \(G/H \times \Disk^\ell \cong G/L \times \Disk^\ell\),
so that it does not really matter whether we use \(H\) or~\(L\) to
describe our equivariant cells.  The dimension~\(d_H\) is defined as
the supremum of all~\(\ell\) such that a cell of type
\(G/L \times \Disk^\ell\) with~\(L\) conjugate to~\(H\) occurs in the
decomposition of~\(X\); this is defined to be~\(-\infty\) if no such
cells exist.  Let
\[
  X^H \defeq \setgiven{x\in X}{h\cdot x = x \text{ for all }h\in H}.
\]
Let~\(X^{(H)}\) be the set difference of~\(X^H\) and~\(X^L\) for all
subgroups \(L\subseteq G\) with \(H\subsetneq L\); this is the set of
points whose stabiliser group is exactly~\(H\).

Let \(\Field\in \{\R,\C,\Quat\}\) and let \(p\colon E\to X\) be a
\(G\)\nb-equivariant \(\Field\)\nb-vector bundle over~\(X\).  Let
\(H\subseteq G\) be a subgroup.  Let \(x\in X^H\).  Then the
\(G\)\nb-action on~\(E\) restricts to an \(H\)\nb-action on the fibre
\(E_x \defeq p^{-1}\{x\}\).  Let~\(\hat{H}_\Field\) denote the set of
isomorphism classes of irreducible representations of~\(H\) on
\(\Field\)\nb-vector spaces.  For \(\varrho\colon G\to \Gl(V)\)
in~\(\hat{H}_\Field\), the commutant \(\Endo(V)\) is one of the
skew-fields \(\R\), \(\C\) or \(\Quat\).  If \(\Field=\R\), then all
three cases can occur, whereas if \(\Field\in\{\C,\Quat\}\), then only
\(\Endo(V)=\Field\) is possible.  Let \(c_\varrho \in \{1,2,4\}\) be
the dimension of \(\Endo(V)\) as an \(\R\)\nb-vector space.  For
\(x\in X^H\) and \(\varrho\in \hat{H}_\Field\), let \(m_\varrho(E_x)\)
be the multiplicity of~\(\varrho\) in this representation of~\(H\)
on~\(E_x\).  The function \(X^H \to \N\), \(x\mapsto m_\varrho(E_x)\),
is locally constant.

The following two theorems are the relative versions of our main
results.  For most applications, the less technical absolute versions
in the two corollaries below will suffice.

\begin{theorem}
  \label{the:summand}
  Let~\(G\) be a compact Lie group and let
  \(\Field\in \{\R,\C,\Quat\}\).  Let \((X,A)\) be a relative
  \(G\)\nb-CW-complex and let \(q\colon V\to X\) and
  \(p\colon E\to X\) be \(G\)\nb-equivariant \(\Field\)\nb-vector
  bundles over~\(X\).  Let \(\iota\colon V|_A \to E|_A\) be an
  injective \(G\)\nb-equivariant \(\Field\)\nb-vector bundle map.
  Assume that for all subgroups \(H\subseteq G\), all
  \(x\in X^{(H)}\setminus A^{(H)}\) and all irreducible
  \(\Field\)\nb-representations \(\varrho\in \hat{H}_\Field\) with
  \(m_\varrho(V_x) \neq 0\), the following inequality holds:
  \[
    m_\varrho(E_x) \ge
    m_\varrho(V_x) + \left\lceil \frac{d_H+1-c_\varrho}{c_\varrho} \right\rceil.
  \]
  Then~\(\iota\) extends to an injective \(G\)\nb-equivariant
  \(\Field\)\nb-vector bundle map \(V \to E\).
\end{theorem}

\begin{theorem}
  \label{the:stable_iso}
  Let~\(G\) be a compact Lie group.  Let \((X,A)\) be a relative
  \(G\)\nb-CW-complex and let \(\Field\in \{\R,\C,\Quat\}\).  Let
  \(q\colon V\to X\) and \(p_j\colon E_j\to X\) for \(j=1,2\) be
  \(G\)\nb-equivariant \(\Field\)\nb-vector bundles over~\(X\).  Let
  \(\varphi_A\colon E_1|_A \congto E_2|_A\) and
  \(\varphi_V\colon E_1 \oplus V \congto E_2 \oplus V\) be
  \(G\)\nb-equivariant \(\Field\)\nb-vector bundle isomorphisms, such
  that \(\varphi_V|_A = \varphi_A \oplus \Id_V\).  Assume that for all
  subgroups \(H\subseteq G\), all \(x\in X^{(H)}\setminus A^{(H)}\)
  and all irreducible \(\Field\)\nb-representations \(\varrho\in \hat{H}_\Field\) with
  \(m_\varrho(V_x) \neq 0\), the following inequality holds:
  \[
    m_\varrho((E_1)_x) \ge \left\lceil \frac{d_H+2-c_\varrho}{c_\varrho} \right\rceil.
  \]
  Then there is a \(G\)\nb-equivariant \(\Field\)\nb-vector bundle
  isomorphism \(\varphi\colon E_1 \congto E_2\) such that
  \(\varphi|_A = \varphi_A\).
\end{theorem}

In both theorems, only subgroups~\(H\) with
\(X^{(H)} \setminus A^{(H)} \neq\emptyset\) occur.  This happens if
and only if the \(G\)-CW-complex decomposition of~\(X\) contains an
equivariant cell of the form \(G/L \times \Disk^\ell\) for some
\(\ell\in \N\) and a subgroup~\(L\) conjugate to~\(H\).

When \(A=\emptyset\), then the data~\(\iota\) in the theorem contains
no information.  So the following are special cases of the theorems:

\begin{corollary}
  \label{cor:summand_absolute}
  Let~\(G\) be a compact Lie group and let
  \(\Field\in \{\R,\C,\Quat\}\).  Let~\(X\) be a \(G\)\nb-CW-complex
  and let \(q\colon V\to X\) and \(p\colon E\to X\) be
  \(G\)\nb-equivariant \(\Field\)\nb-vector bundles over~\(X\).
  Assume that for all subgroups \(H\subseteq G\), all \(x\in X^{(H)}\)
  and all irreducible \(\Field\)\nb-representations \(\varrho\in \hat{H}_\Field\) with
  \(m_\varrho(V_x) \neq 0\), the following inequality holds:
  \[
    m_\varrho(E_x) \ge
    m_\varrho(V_x) + \left\lceil \frac{d_H+1-c_\varrho}{c_\varrho} \right\rceil.
  \]
  Then there is an injective \(G\)\nb-equivariant \(\Field\)\nb-vector
  bundle map \(V \to E\).
\end{corollary}

\begin{corollary}
  \label{cor:stable_iso_absolute}
  Let~\(G\) be a compact Lie group and let
  \(\Field\in \{\R,\C,\Quat\}\).  Let~\(X\) be a \(G\)\nb-CW-complex
  and let \(q\colon V\to X\) and \(p_j\colon E_j\to X\) for \(j=1,2\)
  be \(G\)\nb-equivariant \(\Field\)\nb-vector bundles over~\(X\).
  Assume that there is an isomorphism
  \(E_1 \oplus V \cong E_2 \oplus V\) of \(G\)\nb-equivariant
  \(\Field\)\nb-vector bundles.  For all subgroups \(H\subseteq G\),
  all \(x\in X^{(H)}\) and all irreducible
  \(\Field\)\nb-representations \(\varrho\in \hat{H}_\Field\) with
  \(m_\varrho(V_x) \neq 0\), assume
  \[
    m_\varrho((E_1)_x) \ge \left\lceil \frac{d_H+2-c_\varrho}{c_\varrho} \right\rceil.
  \]
  Then there is a \(G\)\nb-equivariant \(\Field\)\nb-vector bundle
  isomorphism \(E_1 \cong E_2\).
\end{corollary}

The physical Interpretation of Corollary~\ref{cor:stable_iso_absolute}
is as follows.  The space~\(X\) is the Brillouin zone, usually a
torus~\(\T^d\) for a \(d\)\nb-dimensional material.  The vector
bundles \(E_1\) and~\(E_2\) are the Bloch bundles of two topological
insulators that have the same class in equivariant \(\K\)\nb-theory.
The points \(x\in X\) with nontrivial stabiliser~\(H_x\) are the
high-symmetry points in the Brillouin zone.  At these points, the
fibre of the Bloch bundle carries a representation of the
group~\(H_x\), and the irreducible representations of~\(H_x\) are the
symmetry labels of energy bands at the point~\(x\).  The multiplicity
\(m_\varrho((E_1)_x)\) is the number of occupied bands that transform
according to the symmetry~\(\varrho\) at the point~\(x\).  If the
number of occupied bands is sufficiently large, then the topological
phase is robust.  That is, its classification by $\K$\nb-theory is
definitive, meaning that the Bloch vector bundle is determined by its
$\K$\nb-theory class.  The conditions on the multiplicities ensure
that there are enough bands of each symmetry type at the high-symmetry
points to rule out ``fragile'' topological configurations that may be
trivialized by adding more bands.

\section{Proofs of the main theorems}
\label{sec:proofs_main}

The proof follows a standard recipe from equivariant obstruction
theory.  We will first prove a preliminary lemma.  Then we prove
Theorem~\ref{the:summand}.  Finally, we show that
Theorem~\ref{the:summand} implies Theorem~\ref{the:stable_iso}.

\begin{lemma}
  \label{lem:extend_injection}
  Let \(\Field \in\{\R,\C,\Quat\}\) and \(c = \dim_\R(\Field)\).  Let
  \(k\in\N\) and let~\(\Disk^k\) be the \(k\)\nb-dimensional disk.  Let
  \(E\to \Disk^k\) be an \(\Field\)\nb-vector bundle of rank~\(r\).
  Let
  \(\iota\colon \partial \Disk^k \times \Field^m \to E|_{\partial
    \Disk^k}\) be an injective vector bundle map.  If
  \(r-m \ge \lceil(k+1)/c \rceil -1\), then~\(\iota\) extends to an
  injective vector bundle map
  \(\iota'\colon \Disk^k \times \Field^m \to E\).
\end{lemma}

\begin{proof}
  First let \(m=1\).  The map~\(\iota\) on vectors of the form
  \((x,1)\) for \(x\in\partial\Disk^k\) gives a nowhere vanishing
  section of~\(E|_{\partial \Disk^k}\), and~\(\iota'\) exists if and
  only if this extends to a nowhere vanishing section of~\(E\) on all
  of~\(\Disk^k\).  Since~\(\Disk^k\) is contractible, the bundle~\(E\)
  is trivial.  So a nowhere vanishing section is equivalent to a map
  to \(\Field^r\setminus \{0\}\), which is homotopy equivalent
  to~\(\Sphere^{c r -1}\).  Thus our claim becomes equivalent to the
  vanishing of \(\pi_{k-1}(\Sphere^{c r -1})\), which is true if
  \(c r-1 \ge k\) or, equivalently, \(r -m\ge \lceil (k+1)/c\rceil-1\)

  For general \(m\ge1\), the claim is proven by induction over~\(m\).
  First, we may extend the inclusion of the first basis vector
  in~\(\Field^m\) to a nowhere vanishing section of~\(E\) by the
  argument above.  The image of that map is a trivial rank~\(1\)
  \(\Field\)\nb-vector subbundle~\(E_0\) of~\(E\).  The
  quotient~\(E/E_0\) is an \(\Field\)\nb-vector bundle
  over~\(\Disk^k\) of rank~\(r-1\), and~\(\iota\) induces an injective
  vector bundle map from the trivial bundle
  \(\partial \Disk^k \times \Field^{m-1}\) to
  \((E/E_0)|_{\partial \Disk^k}\).  By the induction assumption, the
  latter extends to an injective vector bundle map from
  \(\Disk^k \times \Field^{m-1}\) to \(E/E_0\).  We may lift this to a
  vector bundle map to~\(E\) itself in such a way that we get the
  restriction of~\(\iota\) on \(\partial\Disk^k\times \Field^{m-1}\).
  The resulting vector bundle map \(\Disk^k \times \Field^m \to E\)
  remains injective because the map \(\Disk^k \times \Field \to E_0\)
  is an isomorphism and the map
  \(\Disk^k \times \Field^{m-1}\to E/E_0\) is injective.
\end{proof}

\begin{proof}[Proof of Theorem~\textup{\ref{the:summand}}]
  We extend~\(\iota\) by induction over the skeleta.  So assume that
  we have already extended~\(\iota\) to the \(k-1\)\nb-skeleton,
  consisting of~\(A\) and all equivariant cells
  \(G/H\times \Disk^\ell\) with \(\ell < k\).  We get the
  \(k\)\nb-skeleton from this by attaching a disjoint union of
  equivariant cells of the form \(G/H \times \Disk^k\) along their
  boundaries \(G/H \times \partial \Disk^k\), which are mapped to the
  \(k-1\)-skeleton.  It suffices to build an extension of~\(\iota\) on
  each of these equivariant cells separately.  These may then be put
  together to a continuous vector bundle map on the whole
  \(k\)\nb-skeleton.  And when we can find these extensions on all
  skeleta, then letting \(k\to\infty\) gives a continuous map on
  all of~\(X\).  So our task really is to extend a given
  \(G\)\nb-equivariant injective map between the pull-backs of \(V\)
  and~\(E\) to \(G/H \times \partial \Disk^k\) to a
  \(G\)\nb-equivariant injective map between their pull-backs to
  \(G/H \times \Disk^k\).

  A \(G\)\nb-equivariant map \(\varphi\colon V \to E\) between two
  \(G\)\nb-vector bundles over \(G/H \times Y\) for \(Y= \Disk^k\) or
  \(Y=\partial \Disk^k\) is equivalent to an \(H\)\nb-equivariant map
  between the restrictions of the bundles to
  \(Y \cong \{H\} \times Y \subseteq G/H \times Y\).  Thus we are reduced
  to the problem of extending an injective \(H\)\nb-equivariant vector
  bundle map from \(\partial \Disk^k\) to~\(\Disk^k\).  Here the
  domain and target of the vector bundle map are the pull-backs of
  \(V\) and~\(E\) to~\(Y\) along the canonical map \(Y\to X\).  These
  are \(H\)\nb-equivariant vector bundles, which we denote by \(V'\)
  and~\(E'\).  The inequality in the theorem implies
  \(m_\varrho(V'_y)=0\) or
  \(m_\varrho(E'_y) - m_\varrho(V'_y) \ge \left\lceil
    \frac{k+1}{c_\varrho} \right\rceil -1\) for all \(y\in \Disk^k\)
  and all \(\varrho\in \hat{H}_\Field\).

  Let~\(W\) be a finite-dimensional \(\Field\)\nb-vector space with a
  representation of~\(H\).  By Maschke's Theorem, \(W\) is a direct
  sum of irreducible representations of~\(H\).  We want to make this
  canonical.  Let \(\varrho\colon H \to \Gl(U_\varrho)\) be an
  irreducible representation on a finite-dimensional
  \(\Field\)\nb-vector space~\(U_\varrho\).  Then
  \(\Field_\varrho \defeq \Hom_\Field^H(U_\varrho,U_\varrho)\) is a
  finite-dimensional skew-field by Schur's Lemma, so that
  \(\Field_\varrho \in \{\R,\C,\Quat\}\); if \(\Field\in\{\C,\Quat\}\),
  then only \(\Field_\varrho = \Field\) may occur here.  Let
  \(\Hom_\Field^H(U_\varrho,W)\) denote the set of
  \(H\)\nb-equivariant \(\Field\)\nb-linear maps \(U_\varrho\to W\).
  We turn~\(U_\varrho\) into a left \(\Field_\varrho\)\nb-vector
  space.  Then \(\Hom_\Field^H(U_\varrho,W)\) is a right
  \(\Field_\varrho\)\nb-vector space in a natural way, and there is a
  well-defined, natural, \(H\)\nb-equivariant \(\Field\)\nb-linear map
  \[
    \Hom_\Field^H(U_\varrho,W) \otimes_{\Field_\varrho} U_\varrho \to W,\qquad
    f\otimes u\mapsto f(u).
  \]
  It is zero if~\(W\) does not contain the representation~\(\varrho\).
  It is an isomorphism for \(W=U_\varrho\) by definition, and this
  remains so if~\(W\) is a direct sum of copies of~\(U_\varrho\).
  This gives a canonical \(H\)\nb-equivariant \(\Field\)\nb-linear
  isomorphism
  \begin{equation}
    \label{eq:vector_bundle_structure}
    \bigoplus_{\varrho\in\hat{H}_\Field}
    \Hom_\Field^H(U_\varrho,W) \otimes_{\Field_\varrho} U_\varrho
    \to W.
  \end{equation}
  All this still works if~\(W\) is an \(H\)\nb-equivariant
  \(\Field\)\nb-vector bundle over a space~\(Y\) with trivial
  \(H\)\nb-action.  Then \(\Hom_\Field^H(U_\varrho,W)\) is an
  \(\Field_\varrho\)\nb-vector bundle and
  \(\Hom_\Field^H(U_\varrho,W) \otimes_{\Field_\varrho} U_\varrho\)
  for \(\varrho\in \hat{H}_\Field\) is an \(H\)\nb-equivariant
  \(\Field\)\nb-vector bundle over~\(Y\), and the
  isomorphism~\eqref{eq:vector_bundle_structure} is a natural
  \(H\)\nb-equivariant \(\Field\)\nb-vector bundle isomorphism.  Of
  course, an injective \(H\)\nb-equivariant \(\Field\)\nb-linear map
  \(V'\to E'\) between two \(H\)\nb-equivariant vector bundles induces
  injective \(\Field_\varrho\)\nb-vector bundle maps
  \(\Hom_\Field^H(U_\varrho,V')\to \Hom_\Field^H(U_\varrho,E')\) for
  all \(\varrho\in\hat{H}_\Field\).  Conversely, because of the isomorphism
  above, a family of injective \(\Field\)\nb-vector bundle maps
  \(\Hom_\Field^H(U_\varrho,V')\to \Hom_\Field^H(U_\varrho,E')\) for
  all \(\varrho\in\hat{H}_\Field\) induces an injective \(H\)\nb-equivariant
  \(\Field\)\nb-vector bundle map \(V'\to E'\).

  Recall that our problem is to extend a given injective
  \(H\)\nb-equivariant \(\Field\)\nb-vector bundle map
  \(V'|_{\partial \Disk^k} \to E'|_{\partial \Disk^k}\) to an
  injective \(H\)\nb-equivariant \(\Field\)\nb-vector bundle map
  \(V' \to E'\).  By the equivalence in the previous paragraph, the
  given data is equivalent to a family of injective
  \(\Field_\varrho\)\nb-vector bundle maps
  \[
    \Hom_\Field^H(U_\varrho, V'|_{\partial \Disk^k})\to
    \Hom_\Field^H(U_\varrho, E'|_{\partial \Disk^k})
  \]
  for \(\varrho\in\hat{H}_\Field\), and we must extend each of these to an
  injective \(\Field_\varrho\)\nb-vector bundle map
  \(\Hom_\Field^H(U_\varrho, V')\to \Hom_\Field^H(U_\varrho, E')\).
  Here \(\Hom_\Field^H(U_\varrho,V')\) and
  \(\Hom_\Field^H(U_\varrho,E')\) are \(\Field_\varrho\)\nb-vector
  bundles over~\(\Disk^k\) of rank \(m_\varrho(V'_y)\) and
  \(m_\varrho(E'_y)\), respectively, for any \(y\in \Disk^k\).  There
  is nothing to do if \(m_\varrho(V'_y)=0\), and, otherwise, our
  assumptions imply the inequality
  \(m_\varrho(E'_y) - m_\varrho(V'_y) \ge \left\lceil
    \frac{k+1}{c_\varrho} \right\rceil -1\).  This allows to apply
  Lemma~\ref{lem:extend_injection} to get the desired injective vector
  bundle map.
\end{proof}

\begin{proof}[Proof of Theorem~\textup{\ref{the:stable_iso}}]
  We use our data to define an injective vector bundle map over the
  space \(Y\defeq X\times [0,1]\) relative to the subspace
  \(B \defeq A\times [0,1] \cup X\times \{0,1\}\).  We let \(W\)
  and~\(E\) be the pull-backs of \(V\) and \(E_1 \oplus V\),
  respectively, along the coordinate projection \(Y\to X\),
  \((x,s)\mapsto x\).
  We let \(\iota\colon W\to E\) be the following map, which is clearly
  a \(G\)\nb-equivariant vector bundle map.  If
  \((x,s)\in A\times [0,1] \cup X\times \{0\}\), we let
  \(\iota_{(x,s)}\colon V_x \to (E_1)_x \oplus V_x\) be the obvious
  inclusion map in the second summand.  If \((x,s)\in X\times \{1\}\),
  we let \(\iota_{(x,s)}\colon V_x \to (E_1)_x \oplus V_x\) be the
  composite of the obvious inclusion map
  \(\iota_{(x,s)}\colon V_x \to (E_2)_x \oplus V_x\) with the
  isomorphism
  \((\varphi_V)_x^{-1}\colon (E_2)_x \oplus V_x \congto (E_1)_x \oplus
  V_x\).  The two definitions agree for \((x,0)\) with \(x\in A\)
  because \(\varphi_V|_A = \varphi_A \oplus \Id_V\).  Thus~\(\iota\)
  is a well defined \(G\)\nb-equivariant vector bundle map
  \(W|_B \to E|_B\).  The passage from~\(X\) to~\(X\times [0,1]\)
  increases all dimensions by~\(1\), and the multiplicities for \(W\)
  and~\(E\) are \(m_\varrho(W_{x,s}) = m_\varrho(V_x)\) and
  \(m_\varrho(E_{x,s}) = m_\varrho((E_1)_x) + m_\varrho(V_x)\),
  respectively.  Thus the assumption of this theorem implies the
  condition in Theorem~\ref{the:summand} that guarantees that the
  embedding~\(\iota\) extends to a \(G\)\nb-equivariant
  \(\Field\)\nb-vector bundle embedding \(\iota'\colon W\to E\) over
  all of \(X\times [0,1]\).  Since~\(\iota'\) is a
  \(G\)\nb-equivariant \(\Field\)\nb-vector bundle embedding, its
  cokernel \(E/\iota'(W)\) is again a \(G\)\nb-equivariant
  \(\Field\)\nb-vector bundle over \(X\times [0,1]\).  Its
  restrictions to \(X\times\{0\}\) and \(X\times\{1\}\) are isomorphic
  to \(E_1\) and~\(E_2\), respectively.  So these \(G\)\nb-equivariant
  vector bundles are homotopic.  Homotopic vector bundles are well
  known to be isomorphic, and this remains true for
  \(G\)\nb-equivariant vector bundles (see
  \cite{Wasserman:Equivariant_topology}*{Corollary~2.5}).  Thus we get
  the desired isomorphism.
\end{proof}

\section{Ordinary vector bundles, “real” and “quaternionic” vector
  bundles}
\label{sec:ordinary_real}

In this section, we apply Theorems \ref{the:summand}
and~\ref{the:stable_iso} to equivariant vector bundles over specific
groups that give ordinary vector bundles over the (skew)fields \(\R\),
\(\C\) and~\(\Quat\), or ``real'' and ``quaternionic'' vector bundles.
As we shall see, the results of~\cite{Bakuradze-Meyer:Iso_stable_iso}
are special cases of Theorems \ref{the:summand}
and~\ref{the:stable_iso}.  The
main point is that the relevant groups have the property that all
subgroups that are allowed as stabiliser groups have a unique
irreducible odd representation.  Then the multiplicity of the unique
irreducible representation in a given representation is simply a
quotient of dimensions.  Thus we may replace the multiplicities in
Theorems \ref{the:summand} and~\ref{the:stable_iso} by the ranks of
the vector bundles, multiplied by a suitable constant.

\begin{corollary}
  \label{cor:summand_plain}
  Let \(\Field \in\{\R,\C,\Quat\}\) and \(c = \dim_\R(\Field)\).  Let
  \((X,A)\) be a relative CW-complex of dimension~\(d\).  First, let
  \(q\colon V\to X\) and \(p\colon E\to X\) be \(\Field\)\nb-vector
  bundles.  Let \(\iota\colon V|_A \to E|_A\) be an injective
  \(\Field\)\nb-vector bundle map.  Assume that
  \[
    \dim_\Field E_x \ge
    \dim_\Field V_x + \left\lceil \frac{d+1-c}{c} \right\rceil
  \]
  for all \(x\in X\setminus A\) with \(V_x\neq0\).  Then~\(\iota\)
  extends to an injective \(\Field\)\nb-vector bundle map
  \(V \to E\).

  Secondly, let \(q\colon V\to X\) and \(p_j\colon E_j\to X\) for
  \(j=1,2\) be \(\Field\)\nb-vector bundles over~\(X\).  Let
  \(\varphi_A\colon E_1|_A \congto E_2|_A\) and
  \(\varphi_V\colon E_1 \oplus V \congto E_2 \oplus V\) be
  \(\Field\)\nb-vector bundle isomorphisms such that
  \(\varphi_V|_A = \varphi_A \oplus \Id_V\).  Assume that
  \[
    \dim_\Field E_x \ge
    \left\lceil \frac{d+2-c}{c} \right\rceil
  \]
  for all \(x\in X\setminus A\) with \(V_x\neq0\).  Then there is an
  \(\Field\)\nb-vector bundle isomorphism
  \(\varphi\colon E_1 \congto E_2\) such that \(\varphi|_A = \varphi_A\).
\end{corollary}

Of course, the dimension conditions for all~\(x\) simplify to
\(\rank(E) \ge \rank(V) + \left\lceil \frac{d+1-c}{c} \right\rceil\)
and \(\rank(E) \ge \left\lceil \frac{d+2-c}{c} \right\rceil\),
respectively, if all fibres have the same dimension.  The formulation
above is more general when different fibres have different
dimensions.

\begin{proof}
  This is the special case of Theorems \ref{the:summand}
  and~\ref{the:stable_iso} when~\(G\) is trivial.
\end{proof}

\begin{corollary}
  \label{cor:summand_plain_absolute}
  Let \(\Field \in\{\R,\C,\Quat\}\) and \(c = \dim_\R(\Field)\).
  Let~\(X\) be a CW-complex of dimension~\(d\).  First, let
  \(q\colon V\to X\) and \(p\colon E\to X\) be \(\Field\)\nb-vector
  bundles.  Assume
  \[
    \dim_\Field E_x \ge
    \dim_\Field V_x + \left\lceil \frac{d+1-c}{c} \right\rceil
  \]
  for all \(x\in X\) with \(V_x\neq0\).  Then there is an injective
  \(\Field\)\nb-vector bundle map \(V \to E\).

  Secondly, let \(q\colon V\to X\) and \(p_j\colon E_j\to X\) for
  \(j=1,2\) be \(\Field\)\nb-vector bundles over~\(X\).  Assume
  \(E_1 \oplus V \cong E_2 \oplus V\) and
  \[
    \dim_\Field E_x \ge
    \left\lceil \frac{d+2-c}{c} \right\rceil
  \]
  for all \(x\in X\) with \(V_x\neq0\).  Then there is an
  \(\Field\)\nb-vector bundle isomorphism \(E_1 \cong E_2\).
\end{corollary}

\begin{proof}
  Specialise Corollary~\ref{cor:summand_plain} to the case
  \(A=\emptyset\).
\end{proof}


Corollary~\ref{cor:summand_plain_absolute} predicts that in dimensions
\(d\le 6\), two \(\Quat\)\nb-line bundles are isomorphic once they are
stably isomorphic.  The following example shows that this fails in
dimension~\(7\), so that the dimension threshold in the second part of
Corollary~\ref{cor:summand_plain_absolute} is optimal for
\(\Quat\)\nb-line bundles.  In addition, it follows that
Theorem~\ref{the:line_bundle_stable_iso} below fails for
\(\Quat\)\nb-line bundles.

\begin{example}
  \label{exa:line_bundle_stable_iso}
  On the \(7\)\nb-sphere~\(\Sphere^7\), there are
  \(12\)~isomorphism classes of rank~\(1\) \(\Quat\)\nb-vector
  bundles, all of which are stably isomorphic.

  First we classify the isomorphism classes of rank 1 $\Quat$-bundles
  over $\Sphere^7$.  These are classified by the homotopy group
  $\pi_7(\mathrm{BSp}(1))$.  There is a long exact sequence of
  homotopy groups for the universal fibration
  $\mathrm{Sp}(1) \to \mathrm{ESp}(1) \to \mathrm{BSp}(1)$.  It
  implies an isomorphism
  $\pi_7(\mathrm{BSp}(1)) \cong \pi_6(\mathrm{Sp}(1))$.  The
  group~$\mathrm{Sp}(1)$ of unit quaternions is homeomorphic to the
  \(3\)\nb-sphere~$\Sphere^3$.  The relevant homotopy group
  $\pi_6(\Sphere^3)$ is computed by Toda~\cite{Toda:Composition_homotopy} to
  be $\pi_6(\Sphere^3) \cong \Z/12\Z$.  Thus there are \(12\) isomorphism
  classes of rank~\(1\) \(\Quat\)\nb-vector bundles
  over~\(\Sphere^7\).

  Secondly, we claim that any two \(\Quat\)\nb-vector bundles
  over~\(\Sphere^7\) of the same rank are isomorphic.  The
  Grothendieck group of the monoid of quaternionic vector bundles over
  a space~\(X\) is \(\mathrm{KO}^4(X)\).  We compute the reduced group
  \(\widetilde{\mathrm{KO}}^4(\Sphere^7)\) using Bott Periodicity:
  \[
    \widetilde{\mathrm{KO}}{}^4(\Sphere^7)
    \cong \widetilde{\mathrm{KO}}(\Sphere^{7+4})
    \cong \mathrm{KO}^3(\mathrm{pt})
    \cong 0.
  \]
  As a consequence, any two \(\Quat\)\nb-vector bundles of the same
  rank over~\(\Sphere^7\) are stably isomorphic.  In particular,
  all the~\(12\) nonisomorphic line bundles over~\(\Sphere^7\) are
  stably isomorphic.
\end{example}

To simplify the comparison to the results
in~\cite{Bakuradze-Meyer:Iso_stable_iso}, we restrict attention to
vector bundles of constant rank from now on, that is, we assume that
all their fibres have the same dimension.

\begin{corollary}
  \label{cor:summand_real}
  Let \((X,A)\) be a relative \(\Z/2\)-CW-complex.  Let~\(d_0\) be the
  maximal dimension of trivial cells and~\(d_1\) the maximal dimension
  of free cells in~\((X,A)\).
  \begin{enumerate}
  \item \label{en:summand_real_1}%
    Let \(q\colon V\to X\) and \(p\colon E\to X\) be ``real'' vector
    bundles over~\(X\) of constant rank and let
    \(\iota\colon V|_A \to E|_A\) be an injective ``real'' vector
    bundle map.  Assume
    \[
      \rank(E) \ge \rank(V) + \max \left\{  \left\lceil d_0,
          \frac{d_1-1}{2} \right\rceil \right\}.
    \]
    Then~\(\iota\) extends to an injective ``real'' vector bundle map
    \(V \to E\).
  \item \label{en:summand_real_2}%
    Let \(q\colon V\to X\) and \(p_j\colon E_j\to X\) for \(j=1,2\) be
    ``real'' vector bundles over~\(X\) of constant rank.  Let
    \(\varphi_A\colon E_1|_A \congto E_2|_A\) and
    \(\varphi_V\colon E_1 \oplus V \congto E_2 \oplus V\) be ``real''
    vector bundle isomorphisms such that
    \(\varphi_V|_A = \varphi_A \oplus \Id_V\).  Assume that
    \[
      \rank(E) \ge \max \left\{  \left\lceil d_0+1,
          \frac{d_1}{2} \right\rceil \right\}.
    \]
    Then there is a ``real'' vector bundle isomorphism
    \(\varphi\colon E_1 \congto E_2\) such that
    \(\varphi|_A = \varphi_A\).
  \end{enumerate}
\end{corollary}

\begin{proof}
  Let~\(D_8\) be the dihedral group and let \(t\in D_8\) be rotation
  by~\(\pi\), which is a central involution.  We have already seen in
  Section~\ref{sec:bundles} that a ``real'' vector bundle is the same
  as an odd vector bundle over~\(D_8\) where all rotations act
  trivially on the base~\(X\) and reflections act by~\(\tau\).
  Thus~\((X,A)\) becomes a relative \(D_8\)-CW-complex.  Only two
  subgroups occur as stabilisers, namely, all of~\(D_8\) at the points
  that are fixed by~\(\tau\), and the index-2 subgroup of rotations at
  the points that are not fixed by~\(\tau\).  First let
  \(\tau(x) = x\).  Then the stabiliser group is \(H = D_8\) and
  \(d_H = d_0\).  By Proposition~\ref{pro:unique_odd}, \(H\) has a
  unique irreducible, odd representation~\(\varrho\), namely, the one
  on \(\C\cong \R^2\) mapping the generator~\(a\) to multiplication
  by~\(\ima\) and the other generator~\(b\) to complex conjugation.
  The commutant of this is~\(\R\), so that \(c_\varrho = 1\).  The
  multiplicity of~\(\varrho\) in a representation is the dimension of
  the corresponding \(\C\)\nb-vector space.  Thus
  \(m_\varrho(V_x) = \dim_\C V_x = \rank(V)\) and
  \(m_\varrho(E_x) = \rank(E)\).  Next let \(\tau(x) \neq x\).  Then
  the stabiliser group is the subgroup of rotations
  \(H = \langle a\rangle \cong \Z/4\) and \(d_H = d_1\).  The
  group~\(H\) has a unique irreducible, odd representation~\(\varrho\)
  by Proposition~\ref{pro:unique_odd}, namely, the one on
  \(\C\cong \R^2\) mapping~\(a\) to multiplication by~\(\ima\).  This
  has commutant~\(\C\), so that \(c_\varrho = 2\).  The multiplicity
  of~\(\varrho\) in a representation is the dimension of the
  corresponding \(\C\)\nb-vector space.  Thus
  \(m_\varrho(V_x) = \dim_\C V_x = \rank(V)\) and
  \(m_\varrho(E_x) = \rank(E)\).  Putting together the conditions for
  all \(x\in X\) in Theorems \ref{the:summand}
  and~\ref{the:stable_iso} now gives the conditions in the two
  statements of the corollary.  So Theorems \ref{the:summand}
  and~\ref{the:stable_iso} give the desired conclusions.
\end{proof}


\begin{corollary}
  \label{cor:summand_quaternionic}
  Let \((X,A)\) be a relative \(\Z/2\)-CW-complex.  Let~\(d_0\) be the
  maximal dimension of trivial cells and~\(d_1\) the maximal dimension
  of free cells in~\((X,A)\).
  \begin{enumerate}
  \item Let \(q\colon V\to X\) and \(p\colon E\to X\) be
    ``quaternionic'' vector bundles over~\(X\) of constant rank.  Let
    \(\iota\colon V|_A \to E|_A\) be an injective ``quaternionic''
    vector bundle map.  Assume that
    \[
      \rank(E) \ge \rank(V) + \max \left\{
        2 \left\lceil \frac{d_0-3}{4}\right\rceil,
        \left\lceil \frac{d_1-1}{2}\right\rceil \right\}.
    \]
    Then~\(\iota\) extends to an injective ``quaternionic'' vector
    bundle map \(V \to E\).
  \item Let \(q\colon V\to X\) and \(p_j\colon E_j\to X\) for
    \(j=1,2\) be ``quaternionic'' vector bundles over~\(X\) of
    constant rank.  Let \(\varphi_A\colon E_1|_A \congto E_2|_A\) and
    \(\varphi_V\colon E_1 \oplus V \congto E_2 \oplus V\) be
    ``quaternionic'' vector bundle isomorphisms such that
    \(\varphi_V|_A = \varphi_A \oplus \Id_V\).  Assume that
    \[
      \rank(E) \ge
      \max \left\{
        2 \left\lceil \frac{d_0-2}{4}\right\rceil,
        \left\lceil \frac{d_1}{2}\right\rceil \right\}.
    \]
    Then there is a ``quaternionic'' vector bundle isomorphism
    \(\varphi\colon E_1 \congto E_2\) such that
    \(\varphi|_A = \varphi_A\).
  \end{enumerate}
\end{corollary}

\begin{proof}
  We turn a ``quaternionic'' vector bundle into an odd equivariant
  vector bundle as in Example~\ref{exa:quaternion_group}.  That is,
  \(G=Q_8 \subseteq \Quat\) is the quaternion group with~\(8\)
  elements and \(t= -1 \in Q_8\) is the unique nontrivial central
  element.  We let \(\langle\ima\rangle\) act trivially on~\(X\) and
  let elements of \(Q_8\setminus \langle \ima\rangle\) act
  by~\(\tau\).  Thus~\((X,A)\) becomes a relative \(Q_8\)-CW-complex.
  Only two subgroups occur as stabilisers, namely, all of~\(Q_8\) at
  the points that are fixed by~\(\tau\), and \(\langle \ima\rangle\)
  at the points that are not fixed by~\(\tau\).  First let
  \(\tau(x) = x\).  Then the stabiliser group~\(H\) is~\(Q_8\) and
  \(d_H= d_0\).  By Proposition~\ref{pro:unique_odd}, the
  group~\(Q_8\) has a unique irreducible, odd
  representation~\(\varrho\), namely, the defining representation
  \(Q_8 \hookrightarrow \Quat\).  Its commutant is~\(\Quat\), so that
  \(c_\varrho = 4\).  The multiplicity of~\(\varrho\) in a
  representation is the dimension of the corresponding
  \(\Quat\)\nb-vector space.  Thus
  \(m_\varrho(V_x) = \dim_\C V_x/2 = \rank(V)/2\) and
  \(m_\varrho(E_x) = \rank(E)/2\).  So the conditions at~\(x\) in
  Theorems \ref{the:summand} and~\ref{the:stable_iso} specialise to
  \(\rank(E) \ge \rank(V) + 2 \left\lceil
    \frac{d_0-3}{4}\right\rceil\) and
  \(\rank(E) \ge 2 \left\lceil \frac{d_0-2}{4}\right\rceil\),
  respectively.  Now let \(\tau(x) \neq x\).  Then the stabiliser
  group is \(H = \langle \ima\rangle \cong \Z/4\) and \(d_H = d_1\).
  The group~\(H\) also has a unique irreducible, odd
  representation~\(\varrho\) by Proposition~\ref{pro:unique_odd},
  namely, the obvious one on \(\C\cong \R^2\).  This has
  commutant~\(\C\), so that \(c_\varrho = 2\).  The multiplicity
  of~\(\varrho\) in a representation is the dimension of the
  corresponding \(\C\)\nb-vector space.  Thus
  \(m_\varrho(V_x) = \rank(V)\) and \(m_\varrho(E_x) = \rank(E)\).
  Putting the conditions at~\(x\) in Theorems \ref{the:summand}
  and~\ref{the:stable_iso} together with the condition found for
  \(\tau(x)=x\) now gives the condition with the maximum in this
  corollary.  So our corollary follows from these theorems.
\end{proof}

The previous corollary generalises
\cite{Bakuradze-Meyer:Iso_stable_iso}*{Theorems 3.3 and~3.4} when we
take into account that the rank of a trivial ``quaternionic'' bundle
is always even.

Both Corollaries \ref{cor:summand_real}
and~\ref{cor:summand_quaternionic} have absolute versions for
\(A=\emptyset\), where the piece of data~\(\iota\) is left out.  These
are related to them in the same way that Corollary
\ref{cor:summand_absolute} is related to Theorems \ref{the:summand}
and~\ref{the:stable_iso}.  We leave it to the reader to formulate them.

\section{Subbundles of a trivial bundle}
\label{sec:subbundle_trivial}

Let~\(G\) be a finite group with a central involution~\(t\).
Let~\(X\) be a finite-dimensional \(G\)-CW-complex on which~\(t\) acts
trivially.  Thus~\(X^{(H)}\) is empty for subgroups that do not
contain~\(t\).

In this section, we apply Theorem~\ref{the:summand} to
prove a quantitative version of Swan's Theorem about embedding vector
bundles into a trivial bundle.  Let
\[
  \R[G]_- \defeq
  \setgiven{f\in \R[G]}{f(t g) = -f(g) \text{ for all }g\in G}.
\]
This is the odd part of the regular representation of~\(G\)
over~\(\R\).  The Peter--Weyl Theorem says that \(\R[G]\) is the
direct sum over all irreducible real representations~\(\varrho\)
of~\(G\) with multiplicity \(\dim(\varrho)/c_\varrho\), where
\(\dim(\varrho)\) is the dimension of the underlying \(\R\)\nb-vector
space of~\(\varrho\); so \(\dim(\varrho)/c_\varrho\) is the dimension
of~\(\varrho\) over the field~\(\Field_\varrho\).  Since~\(t\) acts as
\(+1\) in even representations and~\(-1\) in odd irreducible
representations, it follows that \(\R[G]_-\) is an odd representation
that contains each odd irreducible representation~\(\varrho\) of~\(G\)
with the multiplicity \(\dim(\varrho)/c_\varrho\).

The important thing here is that this multiplicity is nonzero.
Analogous results hold if we replace~\(\R[G]_-\) by another odd
representation that contains all odd irreducible representations with
nonzero multiplicity.  (We mention without proof that this implies the
corresponding statement for all subgroups \(H\subseteq G\)
containing~\(t\).)

Let \(H\subseteq G\) be a subgroup containing~\(t\).  Then \(\R[G]_-\)
is the direct sum of \([G:H]\) copies of \(\R[H]_-\).  So~\(\R[G]_-\)
as a representation of~\(H\) is an odd representation that contains
each odd irreducible representation~\(\varrho\) of~\(H\) with the
multiplicity \([G:H]\cdot \dim(\varrho)/c_\varrho\).  Let \(r\in\N\)
and consider the trivial odd \(G\)\nb-equivariant \(\R\)\nb-vector
bundle \(X\times \R[G]_-^r\).  Its fibre at \(x\in X\) contains each
irreducible odd representation~\(\varrho\) of the stabiliser
group~\(G_x\) with multiplicity
\(r\cdot [G: G_x]\cdot \dim(\varrho)/c_\varrho\).

\begin{corollary}[Equivariant Swan's Theorem]
  \label{cor:Swan}
  Let \(V\) be an odd \(G\)\nb-equivariant \(\R\)\nb-vector bundle
  over~\(X\).  Let
  \[
    r\defeq \max \left\{
        \left\lceil
          \frac{c_\varrho m_\varrho(V_x) + d_H + 1 - c_\varrho}
          {[G: G_x]\cdot \dim(\varrho)}
        \right\rceil
        \middle|
        x\in X \text{ and } \varrho\in \hat{G}_x \text{ with }m_\varrho(V_x)\neq 0
      \right\},
  \]
  Then~\(V\) is isomorphic to a direct summand in the trivial bundle
  \(X\times \R[G]_-^r\) .
\end{corollary}

\begin{proof}
  The number~\(r\) is the smallest one for which
  Theorem~\ref{the:summand} ensures that there is a
  \(G\)\nb-equivariant embedding
  \(V\hookrightarrow X\times \R[G]_-^r\).  There is a
  \(G\)\nb-invariant scalar product on~\(\R[G]_-^r\), and the
  orthogonal complement~\(V^\bot\) of the image of~\(V\) is another
  \(G\)\nb-equivariant vector bundle over~\(X\) , such that
  \(V\oplus V^\bot \cong X\times \R[G]_-^r\).
\end{proof}

If we do not care about an optimal value for~\(r\), we could estimate
\(c_\varrho m_\varrho(V_x) \le \dim (V_x)\) and \(d_H \le \dim(X)\).
Since \([G: G_x] \dim(\varrho) \ge1\), the number
\(\rank(V) + \dim(X) + 1\) provides an upper bound
for~\(r\) and so~\(V\) also embeds into the trivial bundle
\(X\times \R[G]_-^{\rank(V) + \dim(X) + 1}\).

Assume now that we are in the special cases considered in
Section~\ref{sec:ordinary_real}.  Then for each stabiliser
group~\(G_x\) there is a unique irreducible odd representation
of~\(G_x\).  This implies \(c_\varrho m_\varrho(V_x) = \dim (V_x)\).
In addition, there are only two subgroups that occur as stabilisers,
and so the expression for~\(r\) simplifies to expressions familiar
from Section~\ref{sec:ordinary_real}.  We refrain from working this
out explicitly.

Still in the special cases considered in
Section~\ref{sec:ordinary_real}, another point is noteworthy: if the
rank of~\(V\) is sufficiently high, then we may write~\(V\) as a
direct sum of a trivial bundle and another bundle~\(V_0\) whose rank
is bounded above by a certain threshold, which is computed in
Section~\ref{sec:ordinary_real} in each case.  Since the rank
of~\(V_0\) is bounded above, Corollary~\ref{cor:Swan} provides an
embedding of~\(V_0\) into a trivial bundle of rank~\(r\) for
some~\(r\) that does not depend on~\(V_0\).  Therefore, for a
sufficiently high~\(r\) depending only on the dimension of~\(X\), any
odd \(G\)\nb-equivariant vector bundle over~\(X\) is a direct sum of a
trivial bundle and of a subbundle of the trivial bundle of rank~\(r\).
We work this out for ``real'' bundles, the other cases being
similar:

\begin{corollary}
  \label{cor:structure_real_bundles}
  Let~\(X\) be a \(\Z/2\)-CW-complex and let \(V\) be a ``real''
  vector bundle over~\(X\) of some constant rank.  Define \(d_0\)
  and~\(d_1\) as in Corollary~\textup{\ref{cor:summand_real}}.  Let
  \[
    r \defeq
    \max \left\{  \left\lceil d_0,
        \frac{d_1-1}{2} \right\rceil \right\}.
  \]
  Then~\(V\) is isomorphic to a direct sum of a trivial bundle and a
  ``real'' vector bundle of rank at most~\(r\).  The latter is a
  direct summand in the trivial ``real'' vector bundle of rank
  \(r+ \min \{r, \rank(V)\} \le 2 r\).

  Assume now that the rank of~\(V\) is at least
  \[
    r_2 \defeq \max \left\{  \left\lceil d_0+1,  \frac{d_1}{2}
      \right\rceil \right\}.
  \]
  Then~\(V\) is a direct sum of a trivial ``real'' vector bundle and a
  subbundle~\(V_0\) of rank~\(r_2\) of the trivial ``real'' vector
  bundle of rank~\(2 r_2\).  Two such vector bundles \(V,V'\) are
  stably isomorphic if and only if the corresponding projections from
  the trivial bundle onto~\(V_0\) are conjugate.
\end{corollary}

\begin{proof}
  Let \(a\defeq \rank V - r\).  If \(a\le 0\), then~\(V\) itself has
  rank at most~\(r\).  Otherwise, let~\(W\) be the trivial bundle of
  rank~\(a\).  Then
  Corollary~\ref{cor:summand_real}.(\ref{en:summand_real_1})
  provides an embedding \(W\hookrightarrow V\), so that
  \(V\cong W\oplus V_0\) for a ``real'' vector bundle of rank~\(r\).
  In all cases, \(V\cong W \oplus V_0\) for a trivial ``real'' vector
  bundle~\(W\) and a ``real'' vector bundle~\(V_0\) of rank
  \(\min \{r, \rank(V)\}\).  This proves the first claim.  Next,
  Corollary~\ref{cor:summand_real}.(\ref{en:summand_real_1})
  implies that~\(V_0\) embeds into the trivial ``real'' bundle of rank
  \(r + \rank(V_0) = r + \min \{r, \rank(V)\} \le 2 r\) as asserted.

  Similar arguments work with the slightly larger rank~\(r_2\)
  instead.  If \(\rank(V) \ge r_2\), then
  Corollary~\ref{cor:summand_real}.(\ref{en:summand_real_2})
  applies both to~\(V_0\) and to its orthogonal
  complement~\(V_0^\bot\) in the trivial bundle of rank~\(2 r_2\).  It
  is noted in~\cite{Bakuradze-Meyer:Iso_stable_iso} that the
  orthogonal projections onto two subbundles \(V_0\) and~\(V_0'\) of
  the trivial ``real'' vector bundle \(X\times \C^{2 r_2}\) are
  conjugate in the algebra
  \[
    \setgiven[\big]{f\in \mathrm{C}(X,\mathbb{M}_{2 r_2}(\C))}
    {f(\tau(z)) = \conj{f(z)}}
  \]
  if and only if both the bundles \(V_0\) and~\(V_0'\) and their
  orthogonal complements are isomorphic.  Since both \(V_0\) and its
  orthogonal complement have rank~\(r_2\), isomorphism is equivalent
  to stable isomorphism for them.  Since
  \([V_0] + [V_0^\bot] = [X\times \C^{2 r_2}] = [V_0'] +
  [(V_0')^\bot]\), a stable isomorphism
  \(V_0 \oplus W \cong V_0' \oplus W\) implies a stable isomorphism
  \[
    (V_0')^\bot \oplus V_0 \oplus W
    \cong (V_0')^\bot \oplus V_0' \oplus W
    \cong (V_0)^\bot \oplus V_0 \oplus W
  \]
  as well.  So a stable isomorphism of \(V_0\) and~\(V_0'\) implies
  that the corresponding projections are conjugate.  The converse is
  trivial.
\end{proof}

\section{Hamiltonians with crystallographic symmetries}
\label{sec:crystallographic}

In this section, we briefly explain the Bloch vector bundle of a
tight-binding Hamiltonian that has only crystallographic symmetries
and apply our main results to this situation.  For a
\(d\)\nb-dimensional material, this is a \(G\)\nb-equivariant vector
bundle over the \(d\)\nb-dimensional torus, where~\(G\) is the point
group of the relevant crystallographic group.

We are very brief about the physical modelling and refer
to~\cite{Prodan-Schulz-Baldes:Bulk_boundary} for more details.  We
work on the Hilbert space \(\ell^2(\Z^d,\C^k)\), where~\(k\) is the
number of internal degrees of freedom in each lattice cell.  We assume
that there is no magnetic field, so that the lattice translations act
simply by \(S_n f(x) = f(x-n)\) for \(x,n\in\Z^d\),
\(f\in\ell^2(\Z^d,\C^k)\).  The translations~\(S_n\) generate the
group \(\Cst\)\nb-algebra \(\Cst(\Z^d)\), which is commutative.  The
Fourier transform identifies \(\Cst(\Z^d)\) with \(\Cont(\T^d)\).

A tight-binding Hamiltonian is a self-adjoint operator on
\(\ell^2(\Z^d,\C^k)\) that commutes with~\(S_n\) for all \(n\in\Z^d\)
and that has ``finite range''.  To explain the latter condition, we
first note that any operator~\(H\) that commutes with all the
translations~\(S_n\) is of the form
\[
  (H f)(x) = \sum_{n\in\Z^d} H_n f(x-n)
\]
with \(H_n\in\Mat_k(\C)\) for \(n\in\Z^d\).  The finite range
assumption means that only finitely many of the matrices~\(H_n\) are
nonzero.  More generally, we may allow~\(H\) to belong to the norm
closure of this set of operators.  This is the same as the tensor
product of \(\Cst(\Z^d)\) acting on \(\ell^2(\Z^d)\) with
\(\Mat_k(\C)\), and the Fourier transform identifies it with
\(\Cont(\T^d,\Mat_k(\C))\).

The material described by the Hamiltonian~\(H\) is an insulator if and
only if~\(H\) is invertible.  (Here we assumed the Fermi energy to be
zero for simplicity.)  Let \(\chi\colon \R\to\{0,1\}\) be the
characteristic function of the negative numbers.  This is continuous
on the spectrum of~\(H\), and \(\chi(H)\) is a projection in the
\(\Cst\)\nb-algebra
\(\Cst(\Z^d) \otimes \Mat_k(\C) \cong \Cont(\T^d,\Mat_k(\C))\).  The
topological phase of the physical system described by~\(H\) is often
defined as the homotopy class of \(\chi(H)\) in the set of projections
in \(\Cont(\T^d,\Mat_k(\C))\).  By the Serre--Swan Theorem,
\(\chi(H)\) corresponds to a vector bundle over~\(\T^d\).  This is the
\emph{Bloch bundle} of~\(H\).  Its fibre at \(z\in\T^d\) is the image
of \(\chi(H)(z) \in \Mat_k(\C)\), and the continuity of the function
\(\chi(H)\) ensures that these subspaces are locally trivial.  The
homotopy class of~\(\chi(H)\) may contain even more information than
the isomorphism class of the Bloch bundle (see
\cite{Bakuradze-Meyer:Iso_stable_iso}*{Section~4}), but we shall focus
on the Bloch bundle in the following.

Now we assume that~\(H\) has some extra crystallographic symmetries.
These are given by an extension
\(L\supseteq \Z^d\) that acts on the lattice~\(\Z^d\).  We denote this
action by \(L \times \Z^d \to \Z^d\), \((l,x)\mapsto \tau_l(x)\).  The
quotient group \(G\defeq L/\Z^d\) is a finite group called the
\emph{point group} of the crystal.  The \(L\)\nb-action on~\(\Z^d\)
induces an action \(\tau_l^*\) on \(\ell^2(\Z^d)\) with
\((\tau_l^*f)(n) \defeq f(\tau_l^{-1} n)\) for all \(l\in L\),
\(n\in\Z^d\), \(f\in \ell^2(\Z^d)\).  This further induces an action
of~\(L\) on the group \(\Cst\)\nb-algebra
\(\Cst(\Z^d) \subseteq \Bound(\ell^2(\Z^d))\) by conjugation.
Since~\(\Z^d\) is Abelian, the translations in~\(L\) commute with
\(\Cst(\Z^d)\).  So the conjugation action on
\(\Cst(\Z^d) \cong \Cont(\T^d)\) factors through an action
\(\alpha\colon G \to \mathrm{Aut}(\Cont(\T^d))\) of the point
group~\(G\).

In addition, our symmetry group~\(L\) also acts on the Hilbert
space~\(\C^k\) of internal degrees of freedom.  We still assume,
however, that translations act trivially on~\(\C^k\), so that we get
some unitary group representation
\(\varrho\colon L \twoheadrightarrow G\to \mathrm{U}(k)\).  Then~\(L\)
acts on \(\ell^2(\Z^d,\C^k) = \ell^2(\Z^d) \otimes \C^k\) by the
representation \(\tau^* \otimes \varrho\).  On the subgroup
\(\Z^d \subseteq L\), this gives the translation operators used above.
The Hamiltonian~\(H\) is assumed to commute with the operators
\(\tau_l^* \otimes \varrho(l)\) for all \(l\in L\).  For \(l\in\Z^d\),
this says that~\(H\) commutes with the translation~\(S_l\) above, but
it gives extra information for \(l\notin \Z^d\).  The conjugation
action of~\(L\) on \(\Cont(\T^d,\Mat_k(\C))\) factors through the
action \(\alpha \otimes \mathrm{Ad}_{\varrho}\) of~\(G\).  So the
crystallographic symmetry means that \(H\in\Cont(\T^d,\Mat_k(\C))\) is
\(G\)\nb-invariant.  Then so is \(\chi(H)\).  And this means that the
Bloch bundle is invariant under the \(G\)\nb-action on the trivial
vector bundle \(\T^d\times (\C^k,\varrho)\).  Thus the Bloch bundle is
a \(G\)\nb-equivariant vector bundle over~\(\T^d\).

Conversely, any \(G\)\nb-equivariant vector bundle over~\(\T^d\) is a
direct summand in a trivial vector bundle by Corollary~\ref{cor:Swan}.
So it is the image of a \(G\)\nb-invariant projection in
\(\Cont(\T^d,\Mat_k(\C))\) for some representation~\(\varrho\)
of~\(G\) on~\(\C^k\).  It is important here to allow nontrivial
representations of~\(G\) on~\(\C^k\).  To see this, assume for the
time being that~\(G\) acts trivially on~\(\C^k\).  Then the induced
action on \(\Cont(\T^d,\Mat_k(\C))\) is only on~\(\T^d\).  So the
fixed-point subalgebra of the \(G\)\nb-action becomes
\(\Cont(\T^d/G,\Mat_k(\C))\).  Thus we get ordinary vector bundles
over the quotient space~\(\T^d/G\).  These are much simpler objects
than \(G\)\nb-equivariant vector bundles over~\(\T^d\).

In order to apply our theorems to a concrete situation, we now
specialise to a particularly simple case: we assume that~\(L\)
consists only of the translations and the point reflection at the
origin, \(n\mapsto -n\).  Thus \(G=\Z/2\).  The induced action on the
torus~\(\T^d\) is the map \(z\mapsto z^{-1} = \conj{z}\).  The Bloch
bundle in this case is a \(\Z/2\)-equivariant \(\C\)\nb-vector bundle
over~\(\T^d\).  This is an absolute \(\Z/2\)-CW-complex
because~\(\T^d\) is a smooth manifold and~\(\Z/2\) acts smoothly on
it.  So the relevant results are Corollaries
\ref{cor:summand_absolute} and~\ref{cor:stable_iso_absolute}.  Here
\(X=\T^d\), \(G=\Z/2\), and \(\Field=\C\), so that only
\(\Field_\varrho=\C\) is possible.  The group~\(\Z/2\) has two
subgroups, the trivial one and~\(\Z/2\) itself.  First let
\(H=\{1\}\).  Then~\(X^{(H)}\) consists of all points in~\(\T^d\)
except the fixed points of the involution.  These are the \(2^d\)
points in \(\{\pm1\}^d \subseteq \T^d\).  So \(d_H= d\).  The
group~\(H\) only has the trivial representation~\(\varrho\) and so
\(m_\varrho(V_x)\) and \(m_\varrho(E_x)\) simplify to the dimension of
\(V_x\) and~\(E_x\) as \(\C\)\nb-vector spaces.  Since~\(\T^d\) is
connected, these two dimensions are the same for all \(x\in\T^d\).
These are just the ranks \(\rank(V)\) and \(\rank(E)\).  Next, let
\(H=\Z/2\).  Then \(X^{(H)} = \{\pm1\}^d\).  So \(d_H= 0\).  Let
\(x\in X^{(H)}\).  The group~\(H\) has only two irreducible
representations, namely, the trivial representation and the sign
character sending the nontrivial element to \(-1\).  We abbreviate
these and write \(m_+\) and~\(m_-\) for the multiplicities of the
trivial and the sign representation, respectively.  The fibres \(V_x\)
and~\(E_x\) are representations of~\(H\).  Since they decompose as a
direct sum of the irreducible representations,
\(m_+(V_x) + m_-(V_x) = \rank(V)\) and
\(m_+(E_x) + m_-(E_x) = \rank(E)\).

\begin{corollary}
  \label{cor:point_reflection}
  Let \(\Z/2\) act on~\(\T^d\) by point reflection.
  \begin{enumerate}
  \item Let \(V\) and~\(E\) be \(\Z/2\)-equivariant \(\C\)\nb-vector
    bundles over~\(\T^d\).  Assume that
    \(\rank(E) \ge \rank(V) + \left\lceil \frac{d-1}{2} \right\rceil\)
    and that \(m_\pm(E_x) \ge m_\pm(V_x)\) for all
    \(x \in \{\pm1\}^d\) and the two signs~\(\pm\).  Then there is an
    injective \(\Z/2\)\nb-equivariant \(\C\)\nb-vector bundle map
    \(V \to E\).
  \item Let \(V\), \(E_1\) and~\(E_2\) be \(\Z/2\)-equivariant
    \(\C\)\nb-vector bundles over~\(\T^d\) such that
    \(E_1 \oplus V \cong E_2 \oplus V\).  Assume that
    \(\rank(E) \ge \left\lceil \frac{d}{2} \right\rceil\).  Then
    \(E_1 \cong E_2\).
  \end{enumerate}
\end{corollary}

\begin{proof}
  This follows by plugging the explicit values of the relevant
  quantities into Corollaries \ref{cor:summand_absolute}
  and~\ref{cor:stable_iso_absolute}.  For \(H=\{\Z/2\}\), the conditions
  simplify because \(d_H=0\) and so
  \(\frac{d_H+2-c_\varrho}{c_\varrho},
  \frac{d_H+1-c_\varrho}{c_\varrho} \le 0\).  We have dropped the
  condition \(m_\varrho((E_1)_x) \ge 0\) because it is always
  satisfied.
\end{proof}

So two stably isomorphic \(\Z/2\)\nb-equivariant \(\C\)\nb-vector
bundles over~\(\T^d\) are isomorphic as soon as their rank is at least
\(\lceil d/2\rceil\).  For \(d \le 2\), this only rules out the
trivial case of rank~\(0\).  For \(d \le 4\), it also rules out
rank~\(1\).  So in all cases except for line bundles, the corollary
above shows that stable isomorphism implies isomorphism.  The same is
true for real or complex line bundles over arbitrary spaces for a
different reason:

\begin{theorem}
  \label{the:line_bundle_stable_iso}
  Let~\(G\) be a compact group.  If two \(G\)-equivariant real or
  complex line bundles are stably isomorphic, then they are
  isomorphic.
\end{theorem}

\begin{proof}
  We write the proof down for complex bundles, the real case is
  analogous.  Example~\ref{exa:line_bundle_stable_iso} shows that the
  result breaks down for quaternionic bundles.  Let \(L, L'\) be
  \(G\)\nb-equivariant complex line bundles over a
  \(G\)\nb-space~\(X\).  Assume that there is a \(G\)\nb-equivariant
  complex vector bundle~\(V\) of rank~\(n\) such that
  \[
    L \oplus V \cong L' \oplus V
  \]
  as complex \(G\)\nb-equivariant vector bundles.  Taking the top
  exterior power of both sides gives
  \[
    \bigwedge^{n+1} {}(L \oplus V) \cong \bigwedge^{n+1} {}(L' \oplus V).
  \]
  Since \(L\) and~\(L'\) are of rank~\(1\) and~\(V\) is of rank~\(n\),
  \[
    \bigwedge^{n+1}{}(L \oplus V) \cong L \otimes \det(V),\qquad
    \bigwedge^{n+1}{}(L' \oplus V) \cong L' \otimes \det(V).
  \]
  for the equivariant determinant line bundle \(\det(V)\).  So we get
  an an equivariant isomorphism
  \[
    L \otimes \det(V) \cong L' \otimes \det(V).
  \]
  Tensoring both sides with the inverse equivariant line bundle
  \(\det(V)^{-1}\) gives \(L\cong L'\) as desired.
\end{proof}

\begin{corollary}
  Stably isomorphic \(\Z/2\)\nb-equivariant \(\C\)\nb-vector bundles
  over~\(\T^d\) are isomorphic if \(d\le 4\).
\end{corollary}

The existence of trivial direct summands is a different matter,
however.  To test this, we would take \(V = \T^d \times \C_\pm\),
where~\(\C_\pm\) denotes~\(\C\) with \(\Z/2\) acting by the trivial or
nontrivial character.  So \(\rank(V)=1\).  If \(d \le 3\), then
\(\rank(E) \ge \rank(V) + \left\lceil \frac{d-1}{2} \right\rceil\)
holds already for \(\rank(E) \ge 2\).  However, we also need the
condition \(m_\pm(E_x) \ge m_\pm(V_x)\) for all \(x \in \{\pm1\}^d\)
and the two signs~\(\pm\).  This means that \(E_x\) must contain the
appropriate representation~\(\C_\pm\) at all \(x\in \{\pm1\}^d\).
This may fail, already for the circle~\(\T^1\).  The following example
shows this already in the simplest case:

\begin{example}
  \label{exa:involutions_and_circle}
  Let \(G=\Z/2\) and let \(X=\Sphere^1\) with the generator
  of~\(G\) acting by complex conjugation, \(z\mapsto \conj{z}\).  We
  claim that there are \(G\)\nb-equivariant complex vector bundles
  over~\(\Sphere^1\) of arbitrarily high rank that do not contain
  any trivial subbundle.  To prove this, we completely classify these
  bundles.

  Any complex vector bundle over~\(\Sphere^1\) is trivial.  After
  trivialising the underlying vector bundle, a \(G\)\nb-equivariant
  complex vector bundle of rank~\(k\) over~\(\Sphere^1\) becomes
  \(\Sphere^1\times \C^k\) with the generator of~\(\Z/2\) acting by
  \((z,v)\mapsto (\conj{z},\Theta_z(v))\) for some linear maps
  \(\Theta_z\in \Gl(k,\C)\) for \(z\in\Sphere^1\), subject to the
  condition that~\(\Theta_z\) is inverse to~\(\Theta_{\conj{z}}\).  In
  particular, \(\Theta_{\pm1}\) are involutions on~\(\C^k\).  In
  addition, \(\Theta_z\) for \(z = x+ \ima \sqrt{1-x^2}\) in the upper
  half circle may be prescribed arbitrarily, and then
  \(\Theta_{x-\ima \sqrt{1-x^2}} \defeq \Theta_{x + \ima
    \sqrt{1-x^2}}^{-1}\) will give~\(\Theta\) on the entire circle.
  Thus a \(\Z/2\)\nb-equivariant vector bundle over~\(\Sphere^1\)
  is represented by two involutions~\(\Theta_{\pm1}\) on~\(\C^k\)
  together with a homotopy between them in \(\Gl(k,\C)\).  We may
  change the trivialisation of the underlying complex vector bundle by
  an arbitrary map \(R\colon \Sphere^1 \to \Gl(k,\C)\), and this
  changes~\(\Theta_z\) to \(R_{\conj{z}} \Theta_z R_z^{-1}\).  In
  particular, we may take \(R_z = 1\) for \(\IM z \ge 0\) and let
  \(R_z\) for \(\IM z \le 0\) be an arbitrary loop in \(\Gl(k,\C)\)
  based at~\(1\).  Since two homotopies between \(\Theta_{\pm1}\)
  differ exactly by such a loop, we see that the choice of the
  homotopy between them does not matter for the isomorphism class of
  the vector bundle.  Since \(\Gl(k,\C)\) is connected, it follows
  that up to isomorphism, the representation is completely determined
  by the two involutions \(\Theta_{\pm1}\).  In fact, these two
  matter only up to conjugacy, and so the only invariant that remains
  are the dimensions~\(\ell_{\pm1}\) of the \(-1\)-eigenspaces
  of~\(\Theta_{\pm1}\).  So for each pair of natural numbers
  \((\ell_{-1},\ell_1)\) between \(0\) and~\(k\), there is exactly one
  \(\Z/2\)-equivariant complex vector bundle over~\(\Sphere^1\).

  Among these bundles, the trivial ones are exactly those with
  \(\ell_{-1} = \ell_1\).  Taking direct sums of equivariant vector
  bundles corresponds to addition of these pairs of numbers.
  Therefore, the vector bundle corresponding to \((\ell_{-1},\ell_1)\)
  has a trivial vector bundle as a direct summand if and only if both
  \(\ell_{-1} \ge1\) and \(\ell_1 \ge1\).  So the equivariant vector
  bundles corresponding to \((0,k)\) and \((k,0)\) do not contain any
  trivial vector subbundle.  Here the rank~\(k\) may be arbitrarily
  large.
\end{example}

As a consequence, the stabilisation map for \(\Z/2\)\nb-equivariant
vector bundles over~\(\T^1\) is not surjective, no matter how high the
rank of the vector bundles.

\smallskip

Next, we consider a system that has both the point reflection
symmetry~\(R\) and a time-reversal symmetry~\(\Theta\).  These must
act by \((R f)(n) = R_0(f(-n))\) and
\((\Theta f)(n) = \Theta_0(f(n))\) for all \(n\in\Z^d\),
\(f\in\ell^2(\Z^d)\), where \(R_0\colon \C^k \to \C^k\) is unitary and
\(\Theta_0\colon \C^k \to \C^k\) is antiunitary.  The maps \(R^2\),
\(\Theta^2\) and \((R\Theta)^2\) are unitary symmetries of the system,
and we assume that they act trivially on the states of the system,
meaning that they are scalar multiples of the identity operator on
\(\ell^2(\Z^d,\C^k)\).  Multiplying~\(R\) by a scalar does not change
the system in a physically observable way.  In this way, we may
arrange that \(R^2=1\).  Since~\(\Theta\) is antiunitary, the scalar
factor in the operator~\(\Theta^2\) is forced to be real because
\(\Theta^2\) commutes with the antiunitary operator~\(\Theta\).  So
\(\Theta^2 = (-1)^a\) for some \(a\in\Z/2\).  Similarly,
\((R \Theta)^2 = (-1)^b\) for some \(b\in\Z/2\).

Systems that only have a time reversal symmetry~\(\Theta\) give two of
the ten fundamental Altland--Zirnbauer symmetry classes that form the
periodic table of topological insulators, where \(\Theta^2 = 1\)
corresponds to the symmetry class AI and \(\Theta^2 = -1\) to the
symmetry class AII.  Systems with both time-reversal and
point-reflection symmetry give symmetry-protected topological phases
that go beyond the standard tenfold way.

Under these assumptions, the set of \(\R\)\nb-linear maps on~\(\C^j\)
\[
  G = \setgiven{\ima ^j}{j\in\Z/4} \cdot \{1, \Theta, R, \Theta R\}
\]
is a group of with~\(16\) elements, and \(t = \ima^2\in G\) is a
central involution.  The map sending \(\ima,\Theta,R\) to
\(\ima,\Theta_0,R_0\) is an odd representation of~\(G\) on~\(\C^k\).
We assume that the Hamiltonian~\(H\) is an invertible operator that
commutes with the action of~\(G\) on \(\ell^2(\Z^d,\C^k)\).  Then the
resulting projection~\(\chi(H)\) also commutes with~\(G\).  This means
that the Bloch bundle is a \(G\)\nb-invariant direct summand of the
trivial \(G\)\nb-equivariant, odd \(\R\)\nb-vector bundle
\(\T^d\times\C^k\).  Here~\(G\) acts on~\(\T^d\) as follows: \(\ima\)
acts trivially and both \(R\) and~\(\Theta\) act by
\(z\mapsto z^{-1} = \conj{z}\).  As a consequence, the stabiliser
group of~\(z\) is all of~\(G\) if \(z \in \{\pm1\}^d\) and the
subgroup~\(H\) of index~\(2\) generated by \(\ima\) and~\(R\Theta\) if
\(z \in \T^d\setminus \{\pm1\}^d\).  The following lemma identifies
\(G\) and this subgroup with the groups~\(G_{p,q}\) introduced above
Proposition~\ref{pro:Cliff_groups}.

\begin{lemma}
  \label{lem:stabilisers_reflection_time}
  There are isomorphisms \(H \cong G_{1,1}\) if \(b=[0]_2\) and
  \(H \cong G_{0,2}\) if \(b=[1]_2\), and \(G \cong G_{2,1}\) if
  \(a=b=[0]_2\), \(G \cong G_{1,2}\) if \(a+b=[1]_2\), and
  \(G \cong G_{0,3}\) if \(a=b=[1]_2\).
\end{lemma}

\begin{proof}
  The subgroup~\(H\) acting on \(\ell^2(\Z^d,\C^k)\) is generated by
  the two anticommuting \(\R\)\nb-linear invertible maps \(\ima\)
  and~\(R\Theta\), which satisfy \(\ima^2=-1\) and
  \((R\Theta)^2 = (-1)^b\).  By definition, this gives~\(G_{p,q}\) as
  in the statement.  The whole group~\(G\) has one extra generator,
  where we may use either \(\Theta\) or~\(\ima \Theta\).  Both
  \(\Theta\) and \(\ima\Theta\) anticommute with~\(\ima\) and satisfy
  \((\ima \Theta)^2 = - \ima^2 \Theta^2 = (-1)^a = \Theta^2\).  If
  \(b=[0]_2\), then~\(\ima \Theta\) anticommutes with~\(\Theta R\).
  If \(b=[1]_2\), then~\(\Theta\) anticommutes with~\(\Theta R\).  In
  either case, \(G\) is generated by three anticommuting operators
  with squares \(-1\), \((-1)^a\), \((-1)^b\).  These generate a
  Clifford algebra and identify~\(G\) with~\(G_{p,q}\) for suitable
  \(p,q\), which are as in the statement of the lemma.
\end{proof}

By Proposition~\ref{pro:Cliff_groups}, the groups \(G_{1,1}\),
\(G_{0,2}\) and~\(G_{1,2}\) have a unique irreducible odd
representation, whereas \(G_{2,1}\) and \(G_{0,3}\) do not.  As a
consequence, for \(a+b=[1]_2\), that is, \((a,b) = ([0]_2,[1]_2)\) or
\((a,b) = ([1]_2,[0]_2)\), there are variants of Corollaries
\ref{cor:summand_real} and~\ref{cor:summand_quaternionic} for
\(G\)\nb-equivariant odd \(\R\)\nb-vector bundles over arbitrary
spaces~\(X\) that use only the ranks of the bundles:

\begin{corollary}
  \label{cor:summand_reflection_time}
  Let \(a,b\in\Z/2\) be such that \(a+b=[1]_2\).  Define the
  groups \(G\) and \(H\subseteq G\) as above.  Let \((X,A)\) be a
  relative \(\Z/2\)-CW-complex, turned into a \(G\)\nb-CW-complex by
  the quotient map \(G \twoheadrightarrow G/H \cong \Z/2\).
  Let~\(d_0\) be the maximal dimension of trivial cells and~\(d_1\)
  the maximal dimension of free cells in~\((X,A)\).
  \begin{enumerate}
  \item \label{en:summand_reflection_time_1}%
    Let \(q\colon V\to X\) and \(p\colon E\to X\) be
    \(G\)\nb-equivariant odd \(\R\)\nb-vector bundles over~\(X\) of
    constant rank.  Let \(\iota\colon V|_A \to E|_A\) be an injective
    \(G\)\nb-equivariant \(\R\)\nb-vector bundle map.  Assume that
    \[
      \rank(E) \ge \rank(V) +
      \begin{cases}
      \max \{ 2 d_0-2, 2 d_1 \}& \text{if }b=[0]_2,\\
      \max \{ 2 d_0-2, d_1 - 3 \}& \text{if }b=[1]_2.
      \end{cases}
    \]
    Then~\(\iota\) extends to an injective \(G\)\nb-equivariant
    \(\R\)\nb-vector bundle map \(V \to E\).
  \item \label{en:summand_reflection_time_2}%
    Let \(q\colon V\to X\) and \(p_j\colon E_j\to X\) for
    \(j=1,2\) be \(G\)\nb-equivariant odd \(\R\)\nb-vector bundles
    over~\(X\) of constant rank.  Let
    \(\varphi_A\colon E_1|_A \congto E_2|_A\) and
    \(\varphi_V\colon E_1 \oplus V \congto E_2 \oplus V\) be
    \(G\)\nb-equivariant \(\R\)\nb-vector bundle isomorphisms such
    that \(\varphi_V|_A = \varphi_A \oplus \Id_V\).  Assume that
    \[
      \rank(E) \ge
      \begin{cases}
        \max \{ 2 d_0, 2 d_1+2 \}& \text{if }b=[0]_2,\\
        \max \{ 2 d_0, d_1 - 2 \}& \text{if }b=[1]_2.
      \end{cases}
    \]
    Then there is a \(G\)\nb-equivariant \(\R\)\nb-vector bundle
    isomorphism \(\varphi\colon E_1 \congto E_2\) such that
    \(\varphi|_A = \varphi_A\).
  \end{enumerate}
\end{corollary}

\begin{proof}
  The unique simple modules over the relevant Clifford algebras
  are~\(\R^2\) for \(\Cliff_{1,1} \cong \Mat_2(\R)\); \(\Quat\) for
  \(\Cliff_{0,2} \cong \Quat\); and~\(\C^2\) for
  \(\Cliff_{1,2} \cong \Mat_2(\C)\).  The unique irreducible odd
  representations of~\(G_{p,q}\) are obtained by restriction, so these
  map the group algebra of~\(G_{p,q}\) to \(\Mat_2(\R)\), \(\Quat\)
  and~\(\Mat_2(\C)\), respectively.  As a consequence, the
  factors~\(c_\varrho\) are \(1,4,2\) in these three cases, and the
  multiplicities of the unique odd irreducible representation in a
  representation in a vector space of \(\R\)\nb-dimension~\(r\) are
  \(r/2\), \(r/4\), and \(r/4\), respectively.  We lose nothing if we
  leave out the ceiling operation because multiplicities are integers
  anyway.  Using the descriptions of the stabiliser groups in
  Lemma~\ref{lem:stabilisers_reflection_time} and plugging the values
  above into Theorems \ref{the:summand} and~\ref{the:stable_iso} now
  gives the statements in the corollary.
\end{proof}

If~\(X\) is the \(d\)\nb-torus with \(d\ge 1\), then \(d_0 = d\ge1\)
and \(d_1=0\).  So the condition in Corollary
\ref{cor:summand_reflection_time}.(\ref{en:summand_reflection_time_2})
simplifies to \(\rank(E) \ge 2 d\).  Since the torus is connected, the
rank is constant.  At the fixed points in \(\{\pm 1\}^d\), the fibre
is a representation of the group \(G = G_{1,2}\), which forces the
rank to be a multiple of~\(4\).  Since rank~\(0\) is trivial, stable
isomorphism implies isomorphism for all bundles for \(d=1\), and for
all bundles except those of rank~\(4\) for \(d=2,3\).

Now let \(a=b=[0]_2\).  Then
\(G= G_{2,1} \subseteq \Cliff_{2,1} \cong\Mat_2(\R)\oplus
\Mat_2(\R)\).  So there are two irreducible odd representations, which
are both of real type and dimension~\(2\).  Similarly, if
\(a=b=[1]_2\), then
\(G= G_{0,3} \subseteq \Cliff_{0,3} \cong \Quat \oplus \Quat\).  So
there are two irreducible odd representations, which are both of
quaternionic type and of dimension~\(4\).  In both cases, it is
impossible to replace multiplicities by ranks at the fixed points in
\(\{\pm1\}^d\).  There is, however, still a unique irreducible
representation of the subgroup~\(H\).  This gives analogues of
Corollary~\ref{cor:point_reflection} for these two cases, which we
leave to the reader to work out.

Our analysis in Corollary~\ref{cor:summand_reflection_time} provides
concrete predictions for $d$-dimensional topological materials with
coexisting point reflection and time-reversal symmetries.  For
instance, consider a 3D system ($X=\T^3$, $d_0=0$, $d_1=3$).  If
$\Theta^2 = +1$ and $R^2 = -1$ (\(a=[0]_2\), \(b=[1]_2\)), then stable
isomorphism implies isomorphism for any bundle of rank greater
than~\(1\), that is, for all vector bundles.  If $\Theta^2 = -1$ and
$R^2 = +1$ (\(a=[1]_2\), \(b=[0]_2\)), then stable isomorphism implies
isomorphism for any bundle of \(\R\)\nb-rank greater than~\(8\).  In
this case, the unique odd irreducible representation at the
high-symmetry points in~\(\T^3\) has dimension~\(4\) over~\(\R\), so
that the rank is always a multiple of~\(4\).  Thus rank~\(4\) is the
only case where there could be a ``fragile'' topological phase, that
is, a topological phase that is stably trivial but not trivial.


\bib{Atiyah:K_Reality}{article}{
  author={Atiyah, Michael F.},
  title={$K$\nobreakdash -theory and reality},
  journal={Quart. J. Math. Oxford Ser. (2)},
  volume={17},
  date={1966},
  pages={367--386},
  issn={0033-5606},
  review={\MR {0206940}},
  doi={10.1093/qmath/17.1.367},
}

\bib{Bakuradze-Meyer:Iso_stable_iso}{article}{
  author={Bakuradze, Malkhaz},
  author={Meyer, Ralf},
  title={Isomorphism and stable isomorphism in ``real'' and ``quaternionic'' K-theory},
  journal={New York J. Math.},
  volume={31},
  pages={690--700},
  issn={1076-9803},
  review={\MR {4895179}},
  eprint={https://nyjm.albany.edu/j/2025/31-25.html},
  date={2025},
}

\bib{Freed-Moore:Twisted_matter}{article}{
  author={Freed, Daniel S.},
  author={Moore, Gregory W.},
  title={Twisted equivariant matter},
  journal={Ann. Henri Poincar\'e},
  volume={14},
  date={2013},
  number={8},
  pages={1927--2023},
  issn={1424-0637},
  review={\MR {3119923}},
  doi={10.1007/s00023-013-0236-x},
}

\bib{Husemoller:Fibre_bundles_3}{book}{
  author={Husemoller, Dale},
  title={Fibre bundles},
  series={Graduate Texts in Mathematics},
  volume={20},
  edition={3},
  publisher={Springer-Verlag},
  place={New York},
  date={1994},
  pages={xx+353},
  isbn={0-387-94087-1},
  review={\MR {1249482}},
}

\bib{Illman:Equivariant_triangulations}{article}{
  author={Illman, S\"oren},
  title={Existence and uniqueness of equivariant triangulations of smooth proper \(G\)\nobreakdash -manifolds with some applications to equivariant Whitehead torsion},
  journal={J. Reine Angew. Math.},
  volume={524},
  date={2000},
  pages={129--183},
  issn={0075-4102},
  review={\MR {1770606}},
  doi={10.1515/crll.2000.054},
}

\bib{DeNittis-Lein:Wannier_zero_flux}{article}{
  author={De Nittis, Giuseppe},
  author={Lein, Max},
  title={Exponentially localized Wannier functions in periodic zero flux magnetic fields},
  journal={J. Math. Phys.},
  volume={52},
  date={2011},
  number={11},
  pages={112103, 32},
  issn={0022-2488},
  review={\MR {2906554}},
  doi={10.1063/1.3657344},
}

\bib{Prodan-Schulz-Baldes:Bulk_boundary}{book}{
  author={Prodan, Emil},
  author={Schulz-Baldes, Hermann},
  title={Bulk and boundary invariants for complex topological insulators},
  series={Mathematical Physics Studies},
  note={From $K$-theory to physics},
  publisher={Springer},
  date={2016},
  pages={xxii+204},
  isbn={978-3-319-29350-9},
  isbn={978-3-319-29351-6},
  review={\MR {3468838}},
  doi={10.1007/978-3-319-29351-6},
}

\bib{Toda:Composition_homotopy}{book}{
  author={Toda, Hirosi},
  title={Composition methods in homotopy groups of spheres},
  series={Annals of Mathematics Studies, No. 49},
  publisher={Princeton University Press, Princeton, N.J.},
  date={1962},
  pages={v+193},
  review={\MR {0143217}},
}

\bib{Wasserman:Equivariant_topology}{article}{
  author={Wasserman, Arthur G.},
  title={Equivariant differential topology},
  journal={Topology},
  volume={8},
  year={1969},
  pages={127--150},
  issn={0040-9383},
  review={\MR {250324}},
  doi={10.1016/0040-9383(69)90005-6},
}



\begin{bibdiv}
  \begin{biblist} \bibselect{references}
  \end{biblist}
\end{bibdiv}
\end{document}